\documentclass[12pt]{article}

\usepackage[left=2cm,right=2cm,top=4cm,bottom=4cm]{geometry}
\usepackage[T1]{fontenc}
\usepackage[latin1]{inputenc}
\usepackage{lmodern}
\usepackage[nottoc, notlof, notlot]{tocbibind}
\usepackage{amsmath}
\usepackage{amsthm}
\usepackage{amssymb}
\usepackage{ stmaryrd }
\usepackage{graphicx}
\usepackage{tikz} 
\usepackage{dsfont}
\usepackage{xspace}
\usepackage{enumitem}
\usepackage{ upgreek }
\usetikzlibrary{matrix,arrows} 
\usepackage[colorlinks,urlcolor=blue,linkcolor=blue,citecolor=brown]{hyperref}

\newcommand{\bbR}{\mathbb{R}}
\newcommand{\bbQ}{\mathbb{Q}}
\newcommand{\bbZ}{\mathbb{Z}}
\newcommand{\bbN}{\mathbb{N}}

\newcommand{\btheta}{\boldsymbol{\theta}}
\newcommand{\homega}{\hat{\omega}}

\newcommand{\bz}{\boldsymbol{z}}
\newcommand{\bq}{\boldsymbol{q}}

\newcommand{\bA}{\boldsymbol{A}}
\newcommand{\bB}{\boldsymbol{B}}

\theoremstyle{plain}
\newtheorem{thm}{Theorem}

\theoremstyle{remark}

\theoremstyle{plain}
\newtheorem{prop}{Proposition}

\theoremstyle{definition}
\newtheorem*{defn}{Definition}

\theoremstyle{plain}
\newtheorem{lem}{Lemma}

\theoremstyle{plain}

\theoremstyle{plain}

\theoremstyle{plain}

\theoremstyle{plain}
\newtheorem*{rmq}{Remark}

\begin{document}

\title{An optimal bound for the ratio between ordinary and uniform exponents of Diophantine approximation}
\author{Antoine Marnat\footnote{supported by Austrian Science Fund (FWF), Project I 3466-N35 and EPSRC Programme Grant EP/J018260/1}  \, \, and Nikolay G. Moshchevitin\footnote{supported by Russian Science Foundation (RNF) Project 18-41-05001 in Pacific National University
}}
 
 \date{}
\maketitle

\abstract{We provide a lower bound for the ratio between the ordinary and uniform exponents of both simultaneous Diophantine approximation 
to $n$ real numbers 
and Diophantine approximation for one linear  form in $n$ variables. This question was first considered in the 50's by V. Jarn\'ik who  solved the problem for two real numbers and established  certain bounds in higher dimension. Recently different authors reconsidered the question, solving the problem in dimension three with different methods. Considering a new concept of  parametric geometry of numbers, W. M. Schmidt and L. Summerer conjectured that the optimal lower bound is reached at regular systems. It follows from a remarkable result of D. Roy that this lower bound is then optimal. In the present paper we give a proof of this conjecture by W. M. Schmidt and L. Summerer.}

\section{Introduction}

In the $50$'s, V. Jarn\'ik \cite{JAR,VJ2,VJ3} considered exponents of Diophantine approximation, and in particular the ratio between ordinary and uniform exponent. An optimal lower bound expressed as a function of the uniform exponent was established for simultaneous approximation to two real numbers and for one linear form in two variables. The question was reconsidered recently by different authors \cite{Laur,M3d,Mo*,SSfirst, GerMo,GM}. The optimality of V. Jarn\'ik's inequalities for two numbers was shown by M. Laurent \cite{Laur}. The inequality for simultaneous approximation to three real numbers was obtained by the second named author \cite{M3d}. Introducing  parametric geometry of numbers \cite{SSfirst,SS}, W. M. Schmidt and L. Summerer considered recently a new method to obtain the optimal lower bounds for the approximation to three numbers (both in the cases of simultaneous approximation and approximation for one linear form in three variables), and improve the general lower bound in any dimension. They conjectured in this context that the general lower bound in the problem of approximation to $n$ real numbers arise from so-called \emph{regular systems}. The goal of the present paper is to prove this conjecture.   To do this  we use Schmidt's inequality on heights \cite{SchLN} applied to a well-chosen subsequence of best approximation vectors. Our main result is stated in Theorem \ref{MainThm} below. The optimality of our bound follows from 
a recent breakthrough paper by D. Roy \cite{Roy}.   \\

Throughout this paper, the integer $n\geq1$ denotes the dimension of the ambient space, and $\btheta=(\theta_1, \ldots , \theta_n)$ denotes an $n$-tuple of real numbers such that $1,\theta_1, \ldots , \theta_n$ are $\mathbb{Q}$-linearly independent.\\ 

Given $n\geq 1$ and $\btheta\in\bbR^n$, we consider the irrationality measure function
\[ \psi(t) = \min_{q\in\bbZ_+,q\leq t} \max_{1 \leq j \leq n}\|q\theta_j \|,  \]
which gives rise to the ordinary exponent of simultaneous Diophantine approximation
\[\lambda(\btheta) = \sup\{ \lambda : \liminf_{t\to + \infty} t^\lambda \psi(t) < + \infty  \} \]
and the uniform exponent of simultaneous Diophantine approximation
\[\hat{\lambda}(\btheta) = \sup\{ \lambda : \limsup_{t\to + \infty} t^\lambda \psi(t) < + \infty  \}.\]
The irrationality measure function
\[ \varphi(t) = \min_{\bq \in \bbZ^n, 0< \max_{1\leq j \leq n} |q_j| \leq t} \|q_1\theta_1 + \cdots + q_n\theta_n \|  \]
 gives rise to the ordinary exponent of Diophantine approximation by one linear form
\[\omega(\btheta) = \sup\{ \omega : \liminf_{t\to + \infty} t^\omega \varphi(t) < + \infty  \} \]
and the uniform exponent of Diophantine approximation by one linear form
\[\hat{\omega}(\btheta) = \sup\{ \omega : \limsup_{t\to + \infty} t^\omega \varphi(t) < + \infty  \}.\]

These exponents were first introduced and studied by A. Khintchine \cite{K2,K1} and V. Jarn\'ik \cite{JAR}. Dirichlet's \emph{Schubfachprinzip} ensures that for any $\btheta$ with $\bbQ$-linearly independent coordinates with $1$
 \[\omega (\btheta) \geq \homega (\btheta) \geq n \textrm{  and } \lambda(\btheta) \geq \hat{\lambda}(\btheta) \geq 1/n.\] 
Exponents of Diophantine approximation give more detailed information about approximation to $\theta$ in the case when $\theta$ admits approximations  better  than the approximations provided by  Dirichlet's \emph{Schubfachprinzip}. The ordinary exponent deals with  the question 
   whether Dirichlet's \emph{Schubfachprinzip} can be improved for approximation vectors of arbitrarily large size $t$, while uniform exponents deals with the  question whether it can be improved for any sufficiently large upper bound $t$ for the size of approximation vectors. The aim of this paper is to provide a lower bound for the ratios $\lambda(\btheta)/\hat{\lambda}(\btheta)$ and $\omega(\btheta)/\hat{\omega}(\btheta)$ as a function of $\hat{\lambda}(\btheta)$ and $\hat{\omega}(\btheta)$ respectively, in any dimension. In dimension $n=1$ simultaneous approximation and approximation by one linear form coincide. Khintchine \cite{K1} observed that the uniform exponent for an irrational $\btheta$ always takes the value $1$ and it follows from Dirichlet's \emph{Schubfachprinzip} that the ordinary exponent satisfy $\omega(\theta) = \lambda(\theta) \geq 1=\hat{\omega}(\theta) = \hat{\lambda}(\theta)$. In dimension $n=2$, Jarn\'ik proved in \cite{VJ2, VJ3} the inequalities
\begin{equation*}
\cfrac{\lambda(\btheta)}{\hat{\lambda}(\btheta)} \geq \cfrac{\hat{\lambda}(\btheta)}{1-\hat{\lambda}(\btheta)}\;\;\textrm{ and }\;\; \cfrac{\omega(\btheta)}{\hat{\omega}(\btheta)} \geq \hat{\omega}(\btheta)-1. 
\end{equation*}
These inequalities are optimal by a result of M. Laurent \cite{Laur}. In \cite{M3d}, Moshchevitin proved the optimal bound for simultaneous approximation in dimension $n=3$:
\begin{eqnarray}\label{Mo3d}
\cfrac{{\lambda}(\btheta)}{\hat{\lambda}(\btheta)} \geq \cfrac{\hat{\lambda}(\btheta) + \sqrt{4\hat{\lambda}(\btheta) - 3\hat{\lambda}(\btheta)^2}}{2(1-\hat{\lambda}(\btheta))} =\frac{1}{2} \left( \frac{\hat{\lambda}(\btheta)}{1-\hat{\lambda}(\btheta)}+ \sqrt{\left( \frac{\hat{\lambda}(\btheta)}{1-\hat{\lambda}(\btheta)} \right)^2 + \frac{4\hat{\lambda}(\btheta)}{1-\hat{\lambda}(\btheta)}}\right).
\end{eqnarray}
The proof is based on consideration of a special pattern of best approximation vectors. This pattern was discovered in an earlier paper by D. Roy \cite{RoyCube}, where another problem was considered. We discuss this pattern in Section \ref{3nb} when explaining our proof in low dimensions.\\

Schmidt and Summerer provided an alternative proof using parametric geometry of numbers in \cite{SS3d}, and found the following bound for approximation by one linear form in $3$ variables:
\begin{equation}\label{SS3omega}
\cfrac{\omega(\btheta)}{\hat{\omega}(\btheta)} \geq \cfrac{\sqrt{4 \hat{\omega}(\btheta)-3} -1}{2}.
\end{equation}
A simple proof of this bound was given in \cite{Mo*}. In \cite{VJ3}, Jarn\'ik  also provided a lower bound in arbitrary dimension $n\geq2$.
\begin{eqnarray}\label{Jomega}
\cfrac{\omega(\btheta)}{\hat{\omega}(\btheta)} &\geq& \hat{\omega}(\btheta)^{1/(n-1)} - 3, \textrm{ provided that } \hat{\omega}(\btheta) > (5n^2)^{n-1},\\ \label{Jlambda}
\cfrac{\lambda(\btheta)}{\hat{\lambda}(\btheta)} &\geq& \cfrac{\hat{\lambda}(\btheta)}{1-\hat{\lambda}(\btheta)}.
\end{eqnarray}
In fact, these bounds also apply in a more general setting of simultaneous Diophantine approximation by a set of linear forms. \\

Using their new tools of parametric geometry of numbers, Schmidt and Summerer \cite{SS} provided the first general improvement valid for the whole admissible interval of values of the uniform exponents $\hat{\omega}$ and $\hat{\lambda}$. 
\begin{eqnarray}\label{SSomega}
\cfrac{\omega(\btheta)}{\hat{\omega}(\btheta)} \geq \cfrac{(n-2)(\hat{\omega}(\btheta)-1)}{1+(n-3)\hat{\omega}(\btheta)},\\ \label{SSlambda}
\cfrac{\lambda(\btheta)}{\hat{\lambda}(\btheta)} \geq \cfrac{\hat{\lambda}(\btheta)+n-3}{(n-2)(1-\hat{\lambda}(\btheta))}.
\end{eqnarray}

Here relation \eqref{SSlambda} is sharper than relation \eqref{Jlambda}. Relation \eqref{SSomega} is  valid for the whole interval of possible values of $\hat{\omega}(\btheta)$, but Jarn\'ik's asymptotic relation \eqref{Jomega} is better for large $\hat{\omega}(\btheta)$. A simple proof of \eqref{SSlambda} was given in \cite{GerMo}.\\ 

In \cite{SS3d} Schmidt and Summerer conjecture that, as in dimension $n=3$, the general optimal lower bound is reached at \emph{regular systems}. In this paper we show that this conjecture holds. Let us first introduce some notation.\\ 

For given $n\geq1$ and parameters $\alpha^* \geq n$ and $1/n \leq \alpha < 1$, we consider the polynomials
\begin{eqnarray}\label{poly} R_{n,\alpha}(x) = x^{n-1} - \cfrac{\alpha}{1-\alpha}(x^{n-2}+ \cdots +x+1),\\ \label{polybis}
R^*_{n,\alpha^*}(x) = x^{n-1}+ x^{n-2} + \cdots +x +1-\alpha^*. \end{eqnarray}
Note that 
$$R_{n,\alpha}(x)= \frac{\alpha }{\alpha-1} \, x^{n-1} R^{*}_{n,\frac{1}{\alpha}}\left(\frac{1}{x}\right).
$$
 Denote by $G(n,\alpha)$ the unique real positive root of $R_{n,\alpha}(x)$ and by $G^*(n,\alpha^*)$ the unique positive root of $R^*_{n,\alpha^*}(x)$. 
 
Some further necessary properties of there polynomials are discussed in Subsection \ref{wwwww} below.

\vskip+0.4cm 
 Now we are able to formulate the main result of our paper.

\begin{thm}\label{MainThm}
For $\btheta=(\theta_1, \ldots , \theta_n)$ such that $1,\theta_1, \ldots , \theta_n$ are $\mathbb{Q}$-linearly independent, one has
\begin{equation}\label{e-Mthm}
\cfrac{\lambda(\btheta)}{\hat{\lambda}(\btheta)} \geq G(n,\hat{\lambda}(\btheta)) \; \textrm{ and } \; 
\cfrac{\omega(\btheta)}{\hat{\omega}(\btheta)} \geq G^*(n,\hat{\omega}(\btheta)).
\end{equation}
Furthermore, for any $\hat{\omega}\geq n$ and any $C\geq G^*(n,\hat{\omega})$, there exists infinitely many $\btheta=(\theta_1, \ldots , \theta_n)$ such that $1,\theta_1, \ldots , \theta_n$ are $\mathbb{Q}$-linearly independent and 
\[ \hat{\omega}(\btheta) = \hat{\omega} \; \; \textrm{ and } \; \; \omega(\btheta) = C\hat{\omega} \]
and for any $1/n \leq \hat{\lambda} \leq 1$ and any $C\geq G(n,\hat{\lambda})$, there exists infinitely many $\btheta=(\theta_1, \ldots , \theta_n)$ such that $1,\theta_1, \ldots , \theta_n$ are $\mathbb{Q}$-linearly independent and 
\[ \hat{\lambda}(\btheta) = \hat{\lambda} \; \; \textrm{ and } \; \; \lambda(\btheta) = C\hat{\lambda}.\]
\end{thm}

It follows from Roy's Theorem \ref{DR} \cite{Roy} applied to Schmidt-Summerer's \emph{regular systems} \cite{SS3d} \cite{Royreg} that the lower bound is reached and thus optimal. The second part of Theorem \ref{MainThm} refines this observation. Note that for any $\btheta=(\theta_1, \ldots , \theta_n)$ such that $1,\theta_1, \ldots , \theta_n$ are $\mathbb{Q}$-linearly independent, we have $\hat{\omega}(\btheta) \geq n$ and $\hat{\lambda}(\btheta) \in [1/n,1]$, (see \cite{Ger}, \cite{Mar}) hence the constraint on $\hat{\lambda}$ and $\hat{\omega}$ is not restrictive. \\

The main part of Theorem \ref{MainThm} is the lower bound. The proof uses determinants of best approximation vectors, following the idea of \cite{M3d}. It deeply relies on an inequality of Schmidt \cite{SchLN} applied inductively to a well chosen subsequence of best approximation vectors. The second part of Theorem \ref{MainThm} is a consequence of the parametric geometry of numbers, and is proved independently in Section \ref{PGN}.\\

In the next section, we define the main tools needed for the proof: best approximation vectors and their properties. With examples of approximation to $3$ and $4$ numbers  in Section \ref{ex34}, we then provide a proof of Theorem \ref{MainThm} in the important case of simultaneous approximation (Section \ref{anynb}). In Section \ref{lf}, we explain how a hyperbolic rotation reduces the case of approximation by one linear form to the case of simultaneous approximation.

\section{Main tools}

\subsection{Sequences of best approximations}

We denote by $(\bz_l)_{l\in\bbN}$ the sequence of best approximations (or minimal points) to $\btheta\in \bbR^n$. This notion was introduced by Voronoi \cite{Vor} as minimal points in lattices, it was first defined in our context by Rogers \cite{Rog}. It has been used implicitly or explicitly in many proofs concerning exponents of Diophantine approximation. Many important properties  of best approximation vectors are discussed in a survey by Chevallier \cite{Chev}.\\

Let $k\geq1$ be an integer. Let $L$ and $N$ be two maps from $\mathbb{Z}^k$ to $\mathbb{R}_+$, where $N$ represents the height of an approximation vector in $\bbZ^{k}$ and $L$ represents the approximation error. We call a sequence of best approximation vectors $(\bz_l)_{l\geq0} \in (\bbZ^k)^\mathbb{N}$ with respect to $L$ and $N$ a sequence such that
\begin{itemize}
\item{$N(\bz_l)$ is a strictly increasing sequence with $N(\bz_1) \ge 1$,}
\item{$L(\bz_l)$ is a strictly decreasing sequence with $L(\bz_1)\le1$,}
\item{for any approximation vector $\bz \in \mathbb{Z}^k$, if $N(\bz) < N(\bz_{l+1})$ then $L(\bz) \geq L(\bz_l)$. }\\
\end{itemize}

In general we do not have uniqueness of such a sequence, and existence follows if $L$ reaches a minimum on sets of the form \[E_B = \{X\in\mathbb{Z}^k | N(X)\leq B \},\] where $B$ is any real bound.\\

In the context of best approximation vectors for simultaneous Diophantine approximation  for $\mathbb{Q}$-independent numbers $1,\theta_1,...,\theta_n$  
we 
set
 \[L_{\lambda,\btheta}(\bz ) = \max_{1\leq i \leq n} |q \theta_i - a_{i} | \; \; \textrm{ and } N_{\lambda,\btheta}(\bz)= q\]
 for
\[ \bz = (q, a_{1}, a_{2}, \ldots , a_{n}) \in \bbZ^{n+1}, l\in \bbN \]  
and define the  unique sequence of best approximation vectors
\[ \bz_l = (q_l, a_{1,l}, a_{2,l}, \ldots , a_{n,l}) \in \bbZ^{n+1}, l\in \bbN \; \textrm{ with } q_l>0\]
with $q_1 =1$.  So $L_{\lambda,\btheta}(\bz_l) = \xi_l = \max_{1\leq i \leq n} |q_l \theta_i - a_{i,l} | $,
$ N_{\lambda,\btheta}(\bz_l)= q_l$ and 
\begin{equation}\label{defBA}
{1 =q_1} < q_2 < \cdots < q_l < q_{l+1}  < \cdots   \;  \textrm{ and } \; 1> \xi_1> \xi_2 >\cdots > \xi_l > \xi_{l+1} > \cdots \, .\end{equation}
We may also assume that for every $l$ large enough one has
\begin{equation}\label{qxi}
\xi_l \leq q_{l+1}^{-\alpha},
\end{equation}
where $\alpha < \hat{\lambda}(\btheta)$.\\

In the context of best approximation vector for approximation by one linear form, we can

 set
 \[ L_{\omega,\btheta}(\bz) = |q_{1}\theta_1 + \cdots + q_{n}\theta_n - a|\: \textrm{ and } \; N_{\omega,\btheta}(\bz) = \max_{1\leq j \leq n} | q_j | \]
 for
\[ \bz = (q_{1}, q_{2}, \ldots , q_{n}, a) \in \bbZ^{n+1} .\]
In such a  way we define the sequence
\[ \bz_l = (q_{1,l}, q_{2,l}, \ldots , q_{n,l}, a_l) \in \bbZ^{n+1},\,\,\,l=1,2,3.. .\]
with
\[ L_{\omega,\btheta}(\bz_l)=  L_l = q_{1,l}\theta_1 + \cdots + q_{n,l}\theta_n - a_l \; \textrm{ and } \; N_{\omega,\btheta}(\bz_l) = M_l = \max_{1\leq j \leq n} | q_j | .\]
{Here due  to the symmetry}, we may assume that $L_l>0$.  
In the  $\mathbb{Q}$-independent 
case this defines vectors $\bz_l $ uniquely. By definition of best approximations 
\[ 1\le M_1 < M_2 < \cdots < M_l <M_{l+1} < \cdots   \;  \textrm{ and } \; 1> L_1> L_2 >\cdots > L_l > L_{l+1} > \cdots\, . \]
We may also assume that $M_1$ is large enough so that for every $l\geq1$
\begin{equation}\label{LM}
L_l \leq M_{l+1}^{-\alpha^*}.
\end{equation}
where $\alpha^* < \hat{\omega}(\btheta)$.\\

In the context of simultaneous Diophantine approximation, provided that $1, \theta_1, \ldots , \theta_n$ are linearly independent over $\bbQ$, it is known that a sequence of best approximation vectors ultimately spans the whole space $\bbR^{n+1}$. However in the context of approximation by one linear form, the situation is different: it may happen that vectors of best approximation span a strictly lower dimensional subspace of $\bbR^{n+1}$. See the surveys \cite{Mo2,Mo1} by Moshchevitin and the paper \cite{Chev} by Chevallier for more details. Fortunately, if best approximation vectors do not span the whole space $\bbR^{n+1}$ we get a sharper result, since $G(n,\alpha)$ is a decreasing function of $n$ (see Proposition \ref{p-decrG}). Thus, we may assume without loss of generality that in both contexts best approximation vectors ultimately span the whole space $\bbR^{n+1}$.\\


Using sequences of best approximation vectors, to prove that $\cfrac{\lambda(\btheta)}{\hat{\lambda}(\btheta)} \geq G$ it is enough to show that given arbitrary $g<G$, there exists arbitrarily large indices $k$ with $q_{k+1} \gg q_k^g$. Similarly, to prove that $\cfrac{\omega(\btheta)}{\hat{\omega}(\btheta)} \geq G^*$ it is enough to show that given arbitrary $g^*<G^*$ and $\alpha^*<\hat{\omega}(\btheta)$, there exists arbitrarily large indices $k$ with $M_{k+1} \gg M_{k}^{g^*}$ or $L_k \ll M_k^{-\alpha^* g^*}$. Here and below, the Vinogradov symbols $\ll$, $\gg$ and $\asymp$ refer to constants depending on $\btheta$ but not the index $k$. This observation relies on the expression of exponents of Diophantine approximation in terms of best approximation vectors

\begin{eqnarray*}
 \omega(\btheta) = \limsup_{k \to \infty}\left(-\cfrac{\log(L_k)}{\log(M_k)}\right) &,& \hat{\omega}(\btheta) = \liminf_{k \to \infty}\left(-\cfrac{\log(L_k)}{\log(M_{k+1})} \right),\\[4mm]
  \lambda(\btheta) = \limsup_{k \to \infty}\left(-\cfrac{\log(\xi_k)}{\log(q_k)}\right) &,& \hat{\lambda}(\btheta) = \liminf_{k \to \infty}\left(\cfrac{\log(\xi_k)}{\log(q_{k+1})}\right).
  \end{eqnarray*}
  For the sake of convenience, if it does not make confusion, we may omit  $(\pmb{\theta})$ in exponents
  $\lambda (\pmb{\theta}), \hat{\lambda}(\pmb{\theta}),
  \omega(\pmb{\theta}), \hat{\omega} (\pmb{\theta})$.

The proofs in the case of simultaneous approximation and approximation by one linear form rely on the same geometric analysis. The idea is to take $\alpha<\hat{\lambda}(\btheta)$ or $\alpha^* < \hat{\omega}(\btheta)$. For an arbitrarily large index $k$, we construct a pattern of best approximation vectors in which at least one pair of successive best approximation vectors satisfies 
\begin{equation}\label{equivthm}
q_{k+1} \gg q_k^g
\end{equation}
where $g= G(n,\alpha) <G(n,\hat{\lambda})$, in the case of simultaneous approximation and 
\begin{equation}\label{equivthmdual}
M_{k+1} \gg M_{k}^{g^*} \;\; \textrm{ or } L_k \ll M_k^{-\alpha^* g^*}
\end{equation}
where $g^*= G^*(n,\alpha^*) <G(n,\hat{\omega})$, in the case of approximation by one linear form.\\
 
 Given a sublattice $\Lambda\subset \mathbb{Z}^{n+1}$, we denote by $\det (\Lambda) $ the fundamental volume of the lattice $\Lambda$ in the linear subspace
 $\langle \Lambda \rangle_\mathbb{R}$. We recall well known facts about best approximation vectors and fundamental determinants of the related lattices.
 \begin{lem}\label{lambdaspans}
Two consecutive best approximation vectors $\bz_i$ and $\bz_{i+1}$ are $\bbQ$-linearly independent and form a basis of the integer points of the rational subspace they span.
\[ \langle \bz_i, \bz_{i+1} \rangle_\bbZ = \langle \bz_i, \bz_{i+1} \rangle_\bbR \cap \bbZ^{n+1}.\]
\end{lem}

See for example \cite[Lemma 2]{DS}.
\begin{lem}\label{estim}
For any $l\geq 1$, consider the lattice $\Lambda_l$ with basis $\bz_l , \bz_{l+1}$ and the lattice $\Gamma_l$ with basis $\bz_{l-1}, \bz_l , \bz_{l+1} $. In the context of simultaneous approximation we have the estimates of their fundamental volumes
\begin{equation}\label{l2}
\det(\Lambda_l) \asymp \xi_l q_{l+1} \;\; \textrm{ and} \;\;  \det( \Gamma_l) \ll \xi_{l-1} \xi_l q_{l+1}.
 \end{equation}
\end{lem}

In the context of approximation by one linear form, we do not have directly such estimates. In section \ref{lf} we explain how hyperbolic rotation provides a helpful analogue.\\

The proof of Lemma \ref{estim} is well known, see for example \cite{Chev} or \cite{Mo1}. For the sake of completeness, and because we want to adapt the proof for the case of approximation by one linear form, we provide a detailed proof. The upper bounds rely on the following lemma (see \cite[Lemma 1]{SchH}), while the lower bound comes from Minkowski's first convex body theorem.
\begin{lem}\label{det2}
Assume $X_1, \ldots ,X_m$ are linearly independent vectors of an Euclidean space $E^n$, and have coordinates $X_t = (x_{t,1}, \ldots , x_{t,n})$ for $1 \leq t \leq m \le n$ in some Cartesian coordinate-system of $E^n$. Then $\det^2(X_1, \ldots , X_m)$ is the sum (with $\binom{m}{n}$ summands) of the squares of the absolute values of the determinants of the $(m \times m)$-submatrices of the matrix $(x_{t,j})_{1 \leq t \leq m, 1 \leq j \leq n}$.
\end{lem}
\begin{proof}[Proof of Lemma \ref{estim}]
The proof relies on the geometric fact that the best approximation $\bz_l = (q_l, a_{1,l}, a_{2,l}, \ldots , a_{n,l}) \in \bbZ^{n+1}$ satisfy \eqref{defBA}. We first prove the upper bounds.\\

Consider the $2$-dimensional fundamental volume of the lattice spanned by $\bz_l, \bz_{l+1}$. The coordinates of these vectors form the matrix
\[\left(\begin{array}{ccccc}q_l & a_{1,l} & \cdots & a_{n,l} \\q_{l+1} & a_{1,l+1} & \cdots  &a_{n,l+1}\end{array}\right).\]
However it is not convenient to use this matrix to apply Lemma \ref{det2}. We consider a special choice of Cartesian coordinates. We take the system of orthogonal unit vector $(e_0, e_1, \ldots , e_n)$ in the following way: $e_0$ is parallel to $(1, \theta_1, \ldots , \theta_n)$ and $e_1, \cdots , e_n$ are arbitrary. Then, in our new coordinates 

\[\bz_l = (Z_{l}, \Xi_{1,l}, \ldots , \Xi_{n,l})\]
where $Z_{l} \asymp q_l$ and $|\Xi_{i,l}| \ll \xi_l$.\\
 Now we consider the $2 \times (n+1)$ matrix
\[\left(\begin{array}{ccccc}Z_{l} & \Xi_{1,l} & \cdots & \Xi_{n,l} \\Z_{l+1} & \Xi_{1,l+1} & \cdots  &\Xi_{n,l+1}\end{array}\right).\]
If $M_{i,j}$ is the $(2 \times 2 )$ minor of index $i,j$, we have by Lemma \ref{det2}
\[\det(\Lambda_l)^2 = \sum_{0\leq i < j \leq n}M_{i,j}^2 \ll \max_{0\leq i < j \leq n}M_{i,j}^2  \ll |Z_{l+1}|^2 \max_{1 \leq i \leq n} |\Xi_{i,l}|^2 \ll (\xi_l q_{l+1})^2.\]

Consider the $3$-dimensional fundamental volume $ \det( \Gamma_l)$ of the lattice spanned by $\bz_{l-1}, \bz_l, \bz_{l+1}$. Denote by $M_{i,j,k}$ the $3\times 3$ minors of the matrix 
\[\left(\begin{array}{ccccc}
Z_{l-1} & \Xi_{1,l-1} & \cdots & \Xi_{n,l-1} \\
Z_{l} & \Xi_{1,l} & \cdots & \Xi_{n,l} \\
Z_{l+1} & \Xi_{1,l+1} & \cdots  &\Xi_{n,l+1}\end{array}\right).\]

By Lemma \ref{det2} we have
\[\det( \Gamma_l)^2 \ll \sum_{0\leq i < j <k \leq n}M_{i,j,k}^2 \ll \max_{0\leq i < j <k\leq n}M_{i,j,k}^2 \ll | Z_{l+1} \Xi_l \Xi_{l-1} |^2 \ll |q_{l+1} \xi_l \xi_{l-1}|^2.\]

We now prove the lower bound for $\det(\Lambda_l)$. Consider the symmetric convex body
 \[\Pi = \{  \bz  \mid |z_0| <  q_{l+1} , \max_{1\leq j \leq n} |z_0 \theta_i - z_i| < \xi_l  \}\]
 and the intersection $\mathcal{P} = \Pi\cap\langle  \bz_l, \bz_{l+1} \rangle_{\bbR}$.   
 The intersection $ \mathcal{P}\cap \langle  \bz_l, \bz_{l+1} \rangle_{\bbZ}$ is reduced to zero by definition of the best approximation.
 Hence Minkowski's first convex body theorem ensures that for the two-dimensional convex set $\mathcal{P}$ we have 
$$
 \textrm{area} ({\mathcal{P}}) \leq 4 \det (\Lambda_l).
 $$
 The intersection of $\mathcal{P}$ with the coordinate hyperplane $ \{ z_0=0\}$ is an interval with endpoints ${\bA}$ and ${\bB}$ of length $ |{\bA}{\bB}| \ge 2\xi_l$.
So  $\mathcal{P}$ contains a polygon $\mathcal{P}'\subset \mathcal{P}$ with vertices ${\bA}, {\bB},   -\bz_{l+1},\bz_{l+1}$.  It is clear that the Euclidean distance between the point  
 $\bz_{l+1}$  and the line ${\bA} {\bB}$ is greater than $q_{l+1}$. We deduce the lower bound for the area of $\mathcal{P}'$
$$
 \textrm{area} ({\mathcal{P}}') \geq 4 q_{l+1}\xi_l.
 $$
 This yields  $q_{l+1}\xi_l\le \det (\Lambda_l)$.
 
 \end{proof}

\paragraph{Notation} We denote by calligraphic letter $\mathcal{S}$ the sets of best approximation vectors $\{ \bz_k, \ldots , \bz_m \}$. Given such a set $\mathcal{S}$, we denote by Greek letter $\Gamma = \langle  \bz_k, \ldots , \bz_m \rangle_{\bbZ}$ the lattice spanned by its elements, and by bold Roman letter $\mathbf{S}= \langle  \bz_k, \ldots , \bz_m \rangle_{\bbR}$ the rational subspace spanned over $\bbR$. Finally, we denote with Gothic letter $\mathfrak{S}$ the underlying lattice of integer points $\mathfrak{S}= \mathbf{S}\cap\bbZ^n$. Note that $ \Gamma \subset \mathfrak{S} $. We should note that two- and  three-dimensional objects play a special role in our proofs (see e.g. Lemma \ref{estim} and Lemma \ref{Keylemma}). Therefore, if our objects are $2$-dimensional, we rather use the letters $\mathcal{L}, \Lambda, \mathbf{L}$ and $\mathfrak{L}$, following notation of previous papers \cite{GM,Mo2,Mo1,M3d,Mo*} dealing with low-dimensional cases. For certain  sets $\mathcal{S}$ 
of consecutive best approximation vectors we will use the word
 \emph{pattern}.
 Fore example three successive independent best approximation vectors 
 $\pmb{z}_{l-1},\pmb{z}_l,\pmb{z}_{l+1}$ form a simplest pattern.
 More complicated patterns may consist of combinations of
  triples of successive best approximation vectors
  connected by certain rules.
  If a pattern $\mathcal{S}$ is the union of say four   patterns $\mathcal{S}_1,\mathcal{S}_2, \mathcal{S}_3$ and $\mathcal{S}_4$, we denoted it by 
\[ \mathcal{S} \quad : \quad \mathcal{S}_1 - \mathcal{S}_2 - \mathcal{S}_3 - \mathcal{S}_4.\]
If moreover the two patterns $\mathcal{S}_2$ and $\mathcal{S}_3$ generate the same rational subspace, we denoted by
\[ \mathcal{S} \quad : \quad \mathcal{S}_1 - \mathcal{S}_2 \equiv \mathcal{S}_3 - \mathcal{S}_4.\]
Finally, if the rational subspaces generated by $\mathcal{S}_1$ and $\mathcal{S}_2$ have intersection $\mathbf{Q}$ and $\mathfrak{Q} = \mathbf{Q}\cap \bbZ^n$ is its lattice of integer points, we denote it by either
\[ \mathcal{S}_1 \underset{{\mathbf{Q}}}{-} \mathcal{S}_2 \quad \textrm{ or } \quad \mathcal{S}_1 \underset{{\mathfrak{Q}}}{-} \mathcal{S}_2. \]

\subsection{Key lemma}
 
 The following lemma plays a key role in the proof of Theorem \ref{MainThm}.
 
 \begin{lem}[$\Gamma_- \underset{\Lambda}{-}\Gamma_+$]\label{Keylemma}
 In the context of simultaneous Diophantine approximation, consider $(\bz_l)_{l\in\bbN}$ the sequence of best approximations to the point $\btheta\in \bbR^n$. 
 Suppose that $ k>\nu$ and  triples
\[ \mathcal{S}_- := \{ \pmb{z}_{\nu-1}, \pmb{z}_\nu, \pmb{z}_{\nu+1} \} \,\,\,\,\text{and} \,\,\,\, \mathcal{S}_+ := \{ \pmb{z}_{k-1}, \pmb{z}_k, \pmb{z}_{k+1}\} \]
consist of linearly independent consecutive best approximation vectors. Suppose that
\begin{equation}\label{zz2}
 \langle  \pmb{z}_\nu, \pmb{z}_{\nu+1} \rangle_\mathbb{Z} =  \langle  \pmb{z}_{k-1}, \pmb{z}_{k} \rangle_\mathbb{Z} =:\Lambda.
 \end{equation}
 and consider the three-dimensional lattices
\[ \mathfrak{S}_- = \langle  \pmb{z}_{\nu-1}, \pmb{z}_\nu, \pmb{z}_{\nu+1} \rangle_\mathbb{R}\cap \mathbb{Z}^n,
 \; \; \textrm{ and } \; \;
   \mathfrak{S}_+ = \langle  \pmb{z}_{k-1}, \pmb{z}_k, \pmb{z}_{k+1} \rangle_\mathbb{R}\cap \mathbb{Z}^n.\]
 
 Suppose that for positive $s$ and $t$ the following estimate holds
 \begin{equation}\label{zz1}
 ({\rm det}\, \mathfrak{S}_-)^s \, ({\det}\, \mathfrak{S}_+)^t \gg {\det}\, \Lambda .
 \end{equation}
Suppose that the index of our vectors are large enough so that for $\alpha < \hat{\lambda}(\btheta)$.\\
 \begin{equation}\label{zz3}
 \xi_j \le q^{-\alpha}_{j+1} \; \; \textrm{ for } \; \;  j = \nu-1,\nu, k -1, k.
 \end{equation}
 Define
\begin{equation}\label{defg} g (s,t) =  \frac{\alpha s}{(1-\alpha)(s-w (s,t))} = \frac{\alpha(t+w(s,t)-1)-w(s,t)+1}{(1-\alpha)t}. \end{equation}
where the second equality comes from $w (s,t)\in (0,1)$ being  the root of the equation
 \begin{equation}\label{eqo}
 w^2 -\left( s+1+\frac{\alpha}{1-\alpha} t\right) w + s = 0.
 \end{equation}
 Assume the positivity condition
  \begin{equation}\label{positivity}
 (1-\alpha)s +t-1 \geq 0.
 \end{equation}
 Then
\begin{equation}\label{zz4}
 \text{either}\,\,\,\,    q_{\nu+1}\gg q_\nu^{g(s,t)}\,\,\,\, \text{or}\,\,\,\,     q_{k+1}\gg q_k^{g(s,t)}.
 \end{equation}
 
 \end{lem}
 
When the parameters are $s=t=1$, this lemma directly provides the result for the approximation to $3$ numbers (Proof from \cite{M3d}, see subsection \ref{3nb} for details). 
Parameters $s$ and $t$ are needed in higher dimension. We exhibit a range of pairs of triples of consecutive best approximation vectors, denoted by an index, satisfying conditions of Lemma \ref{Keylemma}. Parameters $s$ and $t$ appear with values depending on dimension and the geometry of best approximation vectors that need to be optimize with respect to $g(s,t)$. To prove Theorem \ref{MainThm}, we show inductively that the optimized parameter $g(s,t)$ is root of the polynomial $R_{n,\alpha}$ defined by \eqref{poly} for $\alpha<\hat{\lambda}$ arbitrarily close to $\hat{\lambda}$.\\
 
\begin{proof}[Proof of Lemma \ref{Keylemma}] 
From \eqref{eqo} it follows that $s>w(s,t)$ and hence $g>0$. Now we use Lemma \ref{estim}. By \eqref{l2} together with \eqref{zz2}, $\langle  \pmb{z}_{\nu-1}, \pmb{z}_\nu, \pmb{z}_{\nu+1} \rangle_\bbZ \subset \mathfrak{S}_-$ and $\langle  \pmb{z}_{k-1}, \pmb{z}_k, \pmb{z}_{k+1} \rangle_\bbZ \subset \mathfrak{S}_+$, the estimate \eqref{zz1} can be rewritten as:

\[ (\xi_{\nu-1}\xi_\nu q_{\nu+1})^s  (\xi_{k-1}\xi_k q_{k+1})^t  \gg    (\xi_\nu q_{\nu+1})^{w (s,t)}(\xi_{k-1}q_k)^{1-w (s,t)}.\]
  This means that either
\[  (\xi_{\nu-1}\xi_\nu q_{\nu+1})^s  \gg   (\xi_\nu q_{\nu+1})^{w (s,t)}  \; \; \; \; \textrm{ [case (a)]  } \]
  or
 \[   (\xi_{k-1}\xi_k q_{k+1})^t   \gg   (\xi_{k-1}q_k)^{1-w (s,t)}  \; \;\; \; \textrm{ [case (b)]  } .\]
   
   Now we take into account \eqref{zz3}. In case (a), we use $s>w(s,t)$ to deduce that 
   \[1 \ll \xi_{\nu-1}^s(\xi_{\nu}q_{\nu+1})^{s-w(s,t)} \ll q_{\nu}^{-\alpha s}q_{\nu+1}^{(1-\alpha) (s-w(s,t))}\]
   and so $q_{\nu+1} \ll q_{\nu}^{g(s,t)}$. In the case (b) we use condition \eqref{positivity} to deduce  
   \begin{equation}\label{impoo}
   w(s,t)\ge 1-t.
   \end{equation} 
   Indeed,
   consider the function
    \[
   U_{s,t}(w) =
  w^2 -\left( s+1+\frac{\alpha}{1-\alpha} t\right) w + s,
   \]
  which is a polynomial  in $w$  of degree two.
  We see that
   \[
  U_{s,t}(0) = s>0,\,\,\,\,\,\,
  U_{s,t} (1) = -\frac{\alpha}{1-\alpha} t <0.
   \]
  Moreover by  \eqref{positivity}  we have
   \[
  U_{s,t}(1-t) =
  \frac{t}{1-\alpha} ( (1-\alpha)s +t-1)\ge 0
   \]
   and $w(s,t ) \in (0,1)$ is a root of equation $ U_{s,t} (w(s,t)) = 0$. So we get   \eqref{impoo}.
   Now by means of 
   \eqref{impoo}
  we get
   \[1 \ll \xi_{k-1}^{t+w(s,t)-1} \xi_k^t q_{k+1}^t q_k^{w(s,t)-1} \ll  q_k^{w(s,t)-1 - \alpha(t+ w(s,t)-1)}q_{k+1}^{t(1-\alpha)}\]
and  $q_{k+1} \ll q_{k}^{g(s,t)}$.
 \end{proof}
   
\subsection{About the values of  $g(s,t)$}




 This subsection is rather technical and deals with some properties of  $g(s,t)$.
First of all we should note that
 the value of $g=g(s,t)$  defined in Lemma 4 
 satisfies the relation
   \begin{eqnarray}\label{gay}
   g^2 -\left( \frac{\alpha}{1-\alpha} +\frac{1-s}{t}\right) g - \frac{s}{t}\frac{\alpha}{1-\alpha}=0.
      \end{eqnarray}
 Indeed,
  equation \eqref{gay} immediately  follows from \eqref{defg} and \eqref{eqo}.
  

 
Then we should point out that if either $s$ or $t$ is $1$, we can use \eqref{gay} to express back the value of the other parameter $s$ or $t$ in terms of the  value $g=g(s,t)$ defined in (\ref{defg}). Namely,
   we have the following equalities which are equivalent to  \eqref{defg} in the special cases 
   $s=1$ or $t=1$:
   \begin{eqnarray}
      s &=& \frac{g^2-\frac{\alpha}{1-\alpha} g - g}{\frac{\alpha}{1-\alpha} - g}, \; \; \textrm{ for }\;\;  g = g(s,1),\\ \label{guy1}
      t &=& \frac{\frac{\alpha}{1-\alpha} }{g(g-\frac{\alpha}{1-\alpha} )}, \; \; \textrm{ for }\; \;     g = g(1,t), \\  \label{guy2}
      s &=& \frac{g^2-\frac{\alpha}{1-\alpha} g - \frac{\alpha}{1-\alpha}}{g-\frac{\alpha}{1-\alpha} } =\frac{R_{3,\alpha}(g)}{R_{2,\alpha}(g)}, \; \;      \textrm{ for } \; \;   g = g(1-s,1),\\ \label{guy3}
      t &=& \frac{g^2-\frac{\alpha}{1-\alpha} g - \frac{\alpha}{1-\alpha}}{g(g-\frac{\alpha}{1-\alpha}) } =\frac{R_{3,\alpha}(g)}{gR_{2,\alpha}(g)}, \; \; \textrm{ for } \; \;     g = g(1,1-t). 
   \end{eqnarray}

\subsection{About polynomials $R_{n,\alpha}(x)$ and $R^*_{n,\alpha^*}(x)$ }\label{wwwww}

To continue with our exposition we need to establish some further properties of polynomials  $R_{n,\alpha}(x)$, $R^*_{n,\alpha^*}(x)$ defined in \eqref{poly} and \eqref{polybis}.

 \begin{prop}
  The polynomials $R_{n,\alpha}(x)$ and $R^*_{n,\alpha^*}(x)$    can be defined inductively for all $n\geq 2$ {in the following way}:
    
  \begin{equation}\label{aaa}
  \begin{cases}
  R_{2,\alpha}(x) = x - \frac{\alpha}{1-\alpha}  \\R_{n+1,\alpha}(x)= x R_{n,\alpha}(x) - \frac{\alpha}{1-\alpha} 
  \end{cases}, \quad
  \begin{cases}
  R^*_{2,\alpha^*}(x) = x + (1-\alpha^*)  \\ R^*_{n+1,\alpha^*}(x) = R^*_{n,\alpha^*}(x) + x^n
  \end{cases}.
  \end{equation}
  \end{prop}
  
  The result of the proposition above follows from easy calculations.
  
  \vskip+0.3cm Recall that  by 
 $G(n,\alpha)$  we have denoted the unique real positive root of $R_{n,\alpha}(x)$ and by $G^*(n,\alpha^*)$ the unique positive root of $R^*_{n,\alpha^*}(x)$.  
     
   \begin{prop}\label{p-decrG}
 {The values $G(n,\alpha)$ and $G^*(n,\alpha^*)$ are decreasing functions in $n$.}
 \end{prop}
 
 \begin{proof}
 Since $G^*(n,\alpha^*)$ is the unique positive root of $R^*_{n,\alpha^*}$, and $R^*_{n,\alpha^*}(x)\to_{x\to\infty} \infty$, it follows from $R^*_{n+1,\alpha^*}(G^*(n,\alpha^*))= G^*(n,\alpha^*)^n >0 $. The proof is analogous for $G(n,\alpha)$.
 \end{proof}

    \begin{prop} 
 Suppose that $g = G(n,\alpha)$  is the positive root of the polynomial $R_{n,\alpha}(x)$. Then
  \begin{equation} \label{trivial-bounds-g}
      \frac{\alpha}{1-\alpha} \le g \le \frac{1}{1-\alpha}
      \end{equation}  
and 
  \begin{equation}\label{1pos-special-case}
 g((1-\alpha)g-\alpha)\le 1
 \end{equation}

  \end{prop}

  \begin{proof}
  Inequalities  (\ref{trivial-bounds-g}) follow from
   \[
 R_{n,\alpha}\left(\frac{\alpha}{1-\alpha}\right) < 0 < R_{n,\alpha}\left(\frac{1}{1-\alpha}\right).
 \]
  Calculations show that
 \begin{equation} 
 (1-\alpha)\cdot 1 +
 \left( 1-\frac{R_{3,\alpha}(g)}{gR_{2,\alpha}(g)}\right) -1 \ge 0
 \end{equation}
and this is equivalent to  (\ref{1pos-special-case}).
To see this one should take into account 
 the right inequality from  (\ref{trivial-bounds-g}).

  \end{proof}
 
 The following proposition give an analog to  the  inequality  (\ref{trivial-bounds-g}) for the dual case.
 Its proof is quite similar.
  \begin{prop} 
 Suppose that $g^* = G(n,\alpha^*)$  is the positive root of the polynomial $R_{n,\alpha^*}^*(x)$. Then
  \begin{equation} \label{positivitydual}
      \frac{\alpha^*-1}{\alpha^*}<1 \le g^*\le {\alpha^*-1}.
      \end{equation}  
        \end{prop}
 
      \bigskip

\subsection{Schmdt's inequality on heights}

The proof of Theorem \ref{MainThm}  essentially relies on Lemma 4 as well as  on Schmidt's inequality on height (see \cite{SchLN}, in fact this inequality was already used in the last section in \cite{M3d}). It provides the setting to apply Lemma \ref{Keylemma}  for different parameters $(s,t)$ to be determined later.

\begin{prop}[Schmidt's inequality]\label{SchmidtHeight}
Let $A,B$ be two rational subspaces in $\bbR^n$, we have
\begin{equation}\label{SH}
H(A + B) \cdot H(A\cap B) \ll H(A)\cdot H(B).
\end{equation}
where the height $H(A)$ is the fundamental volume of the lattice of integer points $\det(\mathfrak{A})= \det( A \cap \bbZ^n)$.
\end{prop}

\section{Examples: simultaneous approximation to three and four numbers.}\label{ex34}

In this section, we describe in details the proofs in the cases of simultaneous approximation to three and four numbers. 

An example for approximation by one linear form will be presented in Section \ref{4nbdual}.

\subsection{Simultaneous approximation to three numbers}\label{3nb}

Consider $\btheta\in\bbR^3$ with $\bbQ$-linearly independent coordinates with $1$. Consider a sequence $(\bz_l)_{l\in\bbN}$ of best approximations vectors to $\btheta$. Recall that as we consider simultaneous approximation, the sequence $(\bz_l)_{l\in\bbN}$ spans the whole space $\bbR^4$. 
\begin{lem}\label{Lem3nb}
For arbitrarily large indices $k_0$, there exists indices $ k>\nu > k_0$ and  triples of consecutive best approximation vectors
\[ \mathcal{S}_- :=  \{ \pmb{z}_{\nu-1}, \pmb{z}_\nu, \pmb{z}_{\nu+1} \} \; \; \text{and} \; \; \mathcal{S}_+ := \{ \pmb{z}_{k-1}, \pmb{z}_k, \pmb{z}_{k+1} \} \]
consisting of linearly independent vectors such that
  \begin{equation}
\mathfrak{S}_- \cap \mathfrak{S}_+  = : \Lambda = \langle  \pmb{z}_\nu, \pmb{z}_{\nu+1}\rangle_\mathbb{Z} =  \langle  \pmb{z}_{k-1}, \pmb{z}_{k} \rangle_\mathbb{Z}
 \textrm{ and } \; \langle\mathfrak{S}_- \cup \mathfrak{S}_+\rangle_\bbR = \bbR^4,
 \end{equation}
 where
 \[ \mathfrak{S}_- := \langle  \pmb{z}_{\nu-1}, \pmb{z}_\nu, \pmb{z}_{\nu+1} \rangle_\mathbb{R}\cap \mathbb{Z}^n
 \; \; \textrm{ and } \; \;
   \mathfrak{S}_+ := \langle  \pmb{z}_{k-1}, \pmb{z}_k, \pmb{z}_{k+1} \rangle_\mathbb{R}\cap \mathbb{Z}^n.\]
\end{lem}
This was proved in \cite{M3d}. \\

Denote by $\mathcal{S}_4$ the pattern of best approximation vectors described in Lemma \ref{Lem3nb} (see Figure \ref{fig3nb} ). Lemma \ref{Lem3nb} ensures that the pattern
$\mathcal{S}_4$ suites the first conditions to apply Lemma \ref{Keylemma} for arbitrarily large indices.\\

Here we chose $k_0$ sufficiently large for \eqref{zz3} to hold. Schmidt's inequality \eqref{SH} provides \eqref{zz1} with parameters $s=t=1$. Inequality \eqref{positivity} is obvious.\\

Lemma \ref{Keylemma} provides that for any $\alpha < \hat{\lambda}(\btheta)$, 
\[q_{l+1} \gg q_l^{g_\alpha} \]
for $l=\nu$ or $k$, where $g_\alpha$ is solution of the equation \eqref{gay} with $s=t =1$. Namely
\[g_\alpha^2-\frac{\alpha}{1-\alpha}g_\alpha-\frac{\alpha}{1-\alpha} = R_{3,\alpha}(g_{\alpha}) = 0,\] 
which provides 
\[g_\alpha = \cfrac{\alpha + \sqrt{ 4\alpha - 3\alpha^2}}{2(1-\alpha)}.\]
Hence for every $\alpha < \lambda(\btheta)$, we have
\[\cfrac{{\lambda}(\btheta)}{\hat{\lambda}(\btheta)} \geq g_\alpha =  \cfrac{\alpha + \sqrt{ 4\alpha - 3\alpha^2}}{2(1-\alpha)}.\]
We deduce the lower bound \eqref{Mo3d}.\\

We now explain how to obtain the pattern of best approximation vectors in Lemma \ref{Lem3nb}. It is the basic step for a more general construction in higher dimension.

\begin{proof}[Proof of Lemma \ref{Lem3nb}]
Figure \ref{fig3nb} may be usefull to understand the construction.\\
Consider $(\bz_l)_{l\in\bbN}$ a sequence of best approximation vectors to $\btheta\in \bbR^{3}$, and an arbitrarily large index $k_0$. Since $(\bz_l)_{l\geq k_0}$ spans a $4$-dimensional subspace, we can define $k$ to be the smallest index such that
\[ \dim\langle \bz_{k_0}, \bz_{k_0+1}, \ldots, \bz_k, \bz_{k+1} \rangle_\bbR =4.\]
Note that by minimality, $\bz_{k+1}$ is not in the $3$-dimensional subspace spanned by $\left(\bz_l\right)_{k_0 \leq l \leq k}$. In particular, since two consecutive best approximation vectors are linearly independent the three consecutive best approximation vectors $\bz_{k-1}, \bz_{k}, \bz_{k+1}$ are linearly independent. Set $\nu > k_0$ to be the largest index such that
\[ \dim\langle \bz_{\nu-1}, \bz_\nu, \ldots , \bz_k, \bz_{k+1}\rangle_\bbR = 4.\]
Note that by maximality, $\bz_{\nu-1}$ is not in the $3$-dimensional subspace spanned by $\left(\bz_l\right)_{\nu \leq l \leq k +1}$. In particular, since two consecutive best approximation vectors are linearly independent the three consecutive best approximation vectors $\bz_{\nu-1}, \bz_{\nu}, \bz_{\nu+1}$ are linearly independent. Moreover, combining both observations we deduce that the lattice
\[\Lambda := \langle \bz_\nu, \bz_{\nu+1}, \ldots , \bz_{k-1}, \bz_k \rangle_\bbR \cap \bbZ^4 = \langle  \pmb{z}_\nu, \pmb{z}_{\nu+1} \rangle_\mathbb{Z} =  \langle  \pmb{z}_{k-1}, \pmb{z}_{k} \rangle_\mathbb{Z}\] is $2$-dimensional, and is spanned by two consecutive best approximation vectors (see Lemma \ref{lambdaspans}). Hence, the considered indices $\nu$ and $k$ provide $6$ best approximation vectors satisfying Lemma \ref{Lem3nb}.\\
\end{proof}

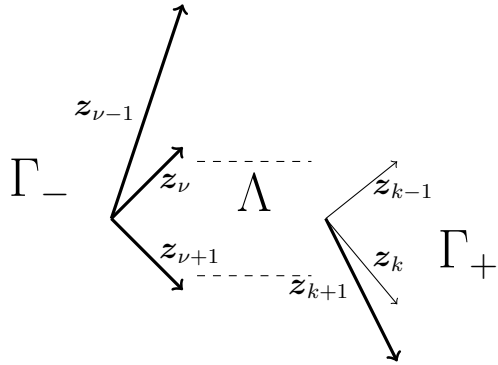
\begin{figure}[!h] 
 \begin{center}
\begin{tikzpicture}[scale=0.95]
\draw[very thick, ->] (0,0)--(1,3) node[midway,left] {$\bz_{\nu-1}$};
\draw[very thick, ->] (0,0)--(1,1) node[midway,right] {$\bz_{\nu}$};
\draw[very thick, ->] (0,0)--(1,-1) node[midway,right] {$\bz_{\nu+1}$};

\draw (-1,0.5) node {\LARGE{$\Gamma_-$}};
\draw (5,-0.5) node {\LARGE{$\Gamma_+$}};

\draw[dashed] (1.2,0.8)--(2.8,0.8) node[midway,below] {\LARGE{$\Lambda$}};
\draw[dashed] (1.2,-0.8)--(2.8,-0.8);

\draw[thin, ->] (3,0)--(4,0.8) node[midway,right] {$\bz_{k-1}$};
\draw[thin, ->] (3,0)--(4,-1.2) node[midway,right] {$\bz_{k}$};
\draw[very thick, ->] (3,0)--(4,-2) ;
\draw (3.1,-1.4) node {$\bz_{k+1}$};

 \end{tikzpicture}
 \end{center}
 \caption{All best approximation vectors with index between $\nu$ and $k$ lie in the $2$-dimensional lattice $\Lambda$. The four bold vectors are linearly independent and span the whole space.}\label{fig3nb}
 \end{figure}

\subsection{Simultaneous approximation to four numbers}\label{4nb}

In the case of simultaneous approximation to four numbers, we select a pattern $\mathcal{S}_5$ of best approximation vectors that combines two patterns $\mathcal{S}_4$ coming from Lemma \ref{Lem3nb}. This is the first step of the induction for arbitrary dimension, where we combine two patterns of lower dimension. Thus, it is an enlightening example. Note that in this simple case, a proper choice of parameters was made in \cite[equalities after formula (13) from the case $\mathfrak{i}(\Theta)=1$]{GM}. \\

Consider $\btheta\in\bbR^4$ with $\bbQ$-linearly independent coordinates with $1$. Consider $(\bz_l)_{l\in\bbN}$ a sequence of best approximation vectors to $\btheta$. 

\begin{lem}\label{lem4nb}
Let $k_0$ be an arbitrarily large index. There exists indices $ k_0 < r_0 < r_1 \leq r_2 < r_3$ such that the following holds.
\begin{enumerate}
\item{ The triples of consecutive best approximation vectors}
\[ \mathcal{S}_{r_i} := \{ \bz_{r_i-1}, \bz_{r_i}, \bz_{r_i+1} \}, \; \; \; 0\leq i \leq 3 \]
consist of linearly independent vectors spanning a $3$-dimensional subspace $\mathbf{S}_{3,i} := \langle \mathcal{S}_{r_i} \rangle_\bbR$.
\item{The two triples of consecutive best approximation vectors} $\mathcal{S}_{r_1}$ and $\mathcal{S}_{r_2}$ generate the same rational subspace 
\[ \mathbf{Q} := \mathbf{S}_{3,1} = \mathbf{S}_{3,2}. \]
\item{The pairs of consecutive best approximation vectors $\bz_{r_0}, \bz_{r_0+1}$ and $\bz_{r_1-1}, \bz_{r_1}$ span the same $2$-dimensional lattice }
\[\Lambda_0 : = \langle \bz_{r_0}, \bz_{r_0+1} \rangle_{\bbZ} = \langle \bz_{r_1-1}, \bz_{r_1} \rangle_{\bbZ} =\mathbf{S}_{3,0} \cap \mathbf{S}_{3,1} \cap \bbZ^{5}.\]
\item{The pairs of consecutive best approximation vectors $\bz_{r_2}, \bz_{r_2+1}$ and $\bz_{r_3-1}, \bz_{r_3}$ span the same $2$-dimensional lattice }
\[\Lambda_1 := \langle \bz_{r_2}, \bz_{r_2+1} \rangle_{\bbZ} = \langle \bz_{r_3-1}, \bz_{r_3}\rangle_{\bbZ}= \mathbf{S}_{3,2}\cap \mathbf{S}_{3,3} \cap \bbZ^5.\]
\item{  Both quadruples of best approximation $\{ \bz_{r_0-1}, \bz_{r_0}, \bz_{r_0+1}, \bz_{r_1+1} \}$ and $\{ \bz_{r_2-1}, \bz_{r_3-1}, \bz_{r_3},\bz_{r_3+1} \}$ consist of linearly independent vectors.}
\item{The whole space $\bbR^5$ is spanned by $\bz_{r_0-1}, \bz_{r_0}, \bz_{r_0+1}, \bz_{r_1+1}, \bz_{r_3+1}$ that is}
\[ \langle \bz_{r_0-1}, \bz_{r_0}, \bz_{r_0+1}, \bz_{r_1+1}, \bz_{r_3+1}\rangle_\bbR = \langle \mathbf{S}_{3,0} \cup \mathbf{Q} \cup \mathbf{S}_{3,2} \rangle_\bbR = \bbR^5.\]
\end{enumerate}
\end{lem}

We discuss the meaning of the lemma, and apply it to the proof of the main result for simultaneous approximation to four numbers. The proof is postponed at the end of the section.\\

The $5$-dimensional pattern described in Lemma \ref{lem4nb} is denoted by 
\[\mathcal{S}_5  \quad : \quad   \mathcal{S}_{3,0} \underset{\Lambda_0}{-}\mathcal{S}_{3,1} \equiv \mathcal{S}_{3,2} \underset{\Lambda_1}{-}\mathcal{S}_{3,3}\]
 till the end of the section. Note that it consists of two $4$-dimensional patterns 
\[\mathcal{S}_{4,0}  \quad : \quad \mathcal{S}_{3,0} \underset{ \Lambda_0 }{-} \mathcal{S}_{3,1}\]
given by indices $\nu=r_0$ and $k=r_1$ in Lemma \ref{Lem3nb} and 
\[ \mathcal{S}_{4,1}  \quad : \quad \mathcal{S}_{3,2} \underset{ \Lambda_1 }{-}  \mathcal{S}_{3,3}\]
given by indices $\nu=r_2$ and $k=r_3$ in Lemma \ref{Lem3nb}. These two $4$-dimensional patterns $\mathcal{S}_{4,0}$ and $\mathcal{S}_{4,1}$ intersect on the $3$-dimensional subspace $\mathbf{Q}$.  Thus,
 \[\mathcal{S}_5  \quad : \quad \mathcal{S}_{4,0} \underset{\mathbf{Q}}{-} \mathcal{S}_{4,1}.\]

 \begin{figure}[!h] 
 \begin{center}
\begin{tikzpicture}[level/.style={sibling distance=50mm/#1, level distance=15mm}]
\node [circle,draw] {${\mathbf{S}}_{5}$}
  child {node [circle,draw] (a) {${\mathbf{S}}_{4,0}$} 
    child {node [circle,draw,black,solid]  (c) {${\mathbf{S}}_{3,0}$} edge from parent  [thin,dotted,black] node[black,sloped] {$\subset$}
    }
    child {node [circle,draw,black,solid] (d) {${\mathbf{S}}_{3,1}$} edge from parent  [thin,dotted,black] node[black,sloped] {$\supset$}
       }
  edge from parent [thin,dotted,black] node[black,sloped] {$\subset$}}
  child {node [circle,draw,black,solid] (b) {${\mathbf{S}}_{4,1}$}
    child {node [circle,draw,black,solid] (e) {${\mathbf{S}}_{3,2}$} edge from parent  [thin,dotted,black] node[black,sloped] {$\subset$}} 
  child {node [circle,draw,black,solid] (f) {${\mathbf{S}}_{3,3}$}
 edge from parent [thin,dotted,black] node[black,sloped] {$\supset$}    }
   edge from parent [thin,dotted,black] node[black,sloped] {$\supset$} };
   \path (d) -- (e) node [midway, sloped] { $=$};
    \path (d) -- (e) node [midway, above] { $\mathbf{Q}$};

 \end{tikzpicture}
 \end{center}
 \caption{Binary tree sketching the situation described in Lemma \ref{lem4nb}.}\label{fignnb}
 \end{figure}
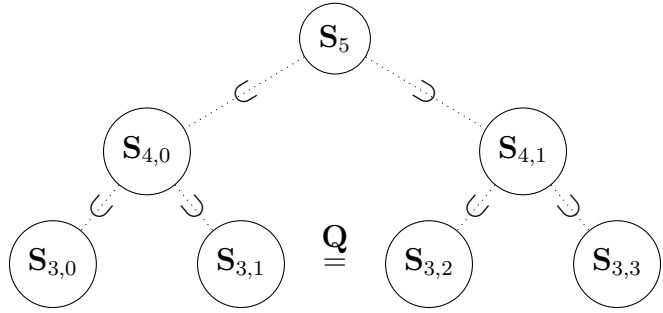

 For the pattern $\mathcal{S}_5$, Schmidt's inequality \eqref{SH} provides 
\[ \det \mathfrak{S}_{3,0} \det \mathfrak{Q} \det \mathfrak{S}_{3,3} \gg \det \Lambda_0 \det \Lambda_1\]
  where $\mathfrak{S}_{i,j}= \mathbf{S}_{i,j} \cap \bbZ^5$ and $\mathfrak{Q}=\mathfrak{S}_{3,1}=\mathfrak{S}_{3,2}$. It can be rewritten  as
 \begin{equation}\label{sh4} \frac{  {\rm det}\, \mathfrak{S}_{3,0} \,  ({\rm det}\, \mathfrak{S}_{3,1})^x}{  {\rm det}\, \Lambda_0}  \cdot   \frac{  ({\rm det}\, \mathfrak{S}_{3,2})^{1-x} \,  {\rm det}\, \mathfrak{S}_{3,3}}{  {\rm det}\, \Lambda_2} \gg 1\end{equation}
 with arbitrary $ x \in (0,1)$. This means that
 \[ \text{either}\,\,\, \frac{  {\rm det}\, \mathfrak{S}_{3,0}\,  ({\rm det}\, \mathfrak{S}_{3,1})^x}{ {\rm det}\, \Lambda_0}\gg 1  \,\,\, \text{or}\,\,\,    \frac{  ({\rm det}\, \mathfrak{S}_{3,2})^{1-x} \,  {\rm det}\, \mathfrak{S}_{3,3}}{
  {\rm det}\, \Lambda_2} \gg 1.\]
  
Hence conditions \eqref{zz2}, \eqref{zz1} and \eqref{zz3} are satisfied either for $(\mathfrak{S}_{3,0}, \mathfrak{S}_{3,1})$ and $(s,t) = (1,x)$ or for $(\mathfrak{S}_{3,2},\mathfrak{S}_{3,3})$ and $(s,t) = (1 - x,1)$.
  For $g$ satisfying the equation $R_{4,\alpha}(g) = gR_{3,\alpha}(g) -\frac{\alpha}{1-\alpha} = 0$, we set 
\[ x = \frac{\frac{\alpha}{1-\alpha}}{g(g-\frac{\alpha}{1-\alpha})} =\frac{R_{3,\alpha}(g)}{g-\frac{\alpha}{1-\alpha}}.\]

From \eqref{guy1}, \eqref{guy2}, we deduce that 
 \begin{equation}\label{pppooo}
 g = G(4,\alpha) = g(1,x) =  g( 1-x,1). 
 \end{equation}
 We should mention that as now $g =G(4,\alpha)$
 is the root of equation $R_{4,\alpha}(x)=0$,
  we have (\ref{trivial-bounds-g}). Hence
  for parameters $(s, t) = (1, x) $ and $(s, t) = (1 - x, 1)$, the positivity condition (\ref{positivity})
 follows 
 from \eqref{trivial-bounds-g}. The first part of Theorem \ref{MainThm} for simultaneous approximation to four numbers follows from Lemma \ref{Keylemma}.\qed
  
 \bigskip
  
  Here, there is one parameter $x$ to optimize. In higher dimension, we have many more, and need to compute the optimal values of that parameters inductively.\\
 
 \begin{proof}[Proof of Lemma \ref{lem4nb}]
Figure \ref{fig4nb} may be usefull to understand the construction.\\
Consider a sequence $(\bz_l)_{l\in\bbN}$ of best approximation vectors to $\btheta\in \bbR^{4}$, and an arbitrarily large index $k_0$. Set $r_3$ to be the smallest index such that
\[ \dim\langle \bz_{k_0}, \bz_{k_0+1}, \ldots, \bz_{r_3}, \bz_{r_3+1} \rangle_\bbR =5.\]
Note that by minimality, $\bz_{r_3+1}$ is not in the $4$-dimensional subspace spanned by $\left(\bz_l\right)_{k_0 \leq l \leq r_3}$. In particular, since two consecutive best approximation vectors are linearly independent the three consecutive best approximation vectors $\bz_{r_3-1}, \bz_{r_3}, \bz_{r_3+1}$ are linearly independent and span a $3$-dimensional lattice denoted by $\Gamma_3$. Set $r_0> k_0$ to be the largest index such that 
\[ \dim\langle \bz_{r_0-1}, \bz_{r_0}, \ldots, \bz_{r_3}, \bz_{r_3+1} \rangle_\bbR =5.\]
Note that by maximality, $\bz_{r_0-1}$ is not in the $4$-dimensional subspace spanned by $\left(\bz_l\right)_{r_0 \leq l \leq r_3 +1}$. In particular, since two consecutive best approximation vectors are linearly independent the three consecutive best approximation vectors $\bz_{r_0-1}, \bz_{r_0}, \bz_{r_0+1}$ are linearly independent and span a $3$-dimensional lattice denoted by $\Gamma_0$. Moreover, combining both observations we deduce that
\[ \mathbf{Q} : = \langle \bz_{r_0}, \bz_{r_0+1}, \ldots , \bz_{r_3-1}, \bz_{r_3} \rangle_\bbR \]
is a $3$-dimensional rational subspace.\\
Now appears the \textbf{induction step}: we apply the same procedure in lower dimension to the two $4$-dimensional subspaces 
\[ \mathbf{S}_{4,0} := \langle \bz_{r_0-1}, \bz_{r_0}, \ldots , \bz_{r_3-1}, \bz_{r_3}\rangle_\bbR \; \; \textrm{ and } \; \;  \mathbf{S}_{4,1} :=   \langle \bz_{r_0}, \bz_{r_0+1}, \ldots , \bz_{r_3}, \bz_{r_3+1}\rangle_\bbR .\]
Note that it gives a proof of  Lemma \ref{Lem3nb}.\\
Set $r_1$ to be the smallest index such that 
\[\langle \bz_{r_0-1}, \bz_{r_0}, \ldots , \bz_{r_1}, \bz_{r_1+1}\rangle_\bbR = \mathbf{S}_{4,0}.\]
Note that by minimality, $\bz_{r_1+1}$ is not in the $3$-dimensional subspace $\mathbf{S}_{3,0}$ spanned by $\left(\bz_l\right)_{r_0-1 \leq l \leq r_1}$. In particular, since two consecutive best approximation vectors are linearly independent the three consecutive best approximation vectors $\bz_{r_1-1}, \bz_{r_1}, \bz_{r_1+1}$ are linearly independent and span a $3$-dimensional lattice $\Gamma_1$ included in $\mathbf{Q} = \mathbf{S}_{3,1}$. By construction, $r_0$ is already the largest index such that 
\[\langle  \bz_{r_0-1}, \bz_{r_0}, \ldots , \bz_{r_1-1}, \bz_{r_1}\rangle_\bbR = \mathbf{S}_{4,0}.\]
Hence, $\langle \bz_{r_0}, \bz_{r_0+1}, \ldots , \bz_{r_1-1}, \bz_{r_1}\rangle_\bbZ= : \Lambda_0$ is a $2$-dimensional lattice spanned by either $\langle \bz_{r_0}, \bz_{r_0+1} \rangle_{\bbZ}$ or $\langle \bz_{r_1-1}, \bz_{r_1} \rangle_{\bbZ}$, and is the intersection $\mathbf{S}_{3,0} \cap \mathbf{S}_{3,1} \cap \bbZ^5$ (see Lemma \ref{lambdaspans}). \\
Set $r_2$ to be the largest index such that 
\[ \langle \bz_{r_2-1}, \bz_{r_2}, \ldots , \bz_{r_3}, \bz_{r_3+1} \rangle_\bbR = \mathbf{S}_{4,1}.\]
Note that $\bz_{r_2-1}$ is not in the $3$-dimensional subspace $\mathbf{S}_{3,3}$ spanned by $\left(\bz_l\right)_{r_2 \leq l \leq r_3+1}$. In particular, since two consecutive best approximation vectors are linearly independent the three consecutive best approximation vectors $\bz_{r_2-1}, \bz_{r_2}, \bz_{r_2+1}$ are linearly independent and span a $3$-dimensional lattice $\Gamma_2$ included in $\mathbf{Q} = \mathbf{S}_{3,1}$. By construction, $r_3$ is already the smallest index such that 
\[\langle \bz_{r_2-1}, \bz_{r_2}, \ldots , \bz_{r_3}, \bz_{r_3+1}\rangle_\bbR = \mathbf{S}_{4,1}.\]
Hence, $\langle \bz_{r_2}, \bz_{r_2+1}, \ldots , \bz_{r_3-1}, \bz_{r_3}\rangle_\bbZ=: \Lambda_1$ is a $2$-dimensional lattice spanned by $\langle \bz_{r_2}, \bz_{r_2+1} \rangle_{\bbZ}$ or $\langle \bz_{r_3-1}, \bz_{r_3} \rangle_{\bbZ}$, and is the intersection $\mathbf{Q} \cap \mathbf{S}_{3,3} \cap \bbZ^5$ (see Lemma \ref{lambdaspans}).
\end{proof}
Note that we may have $r_1=r_2$. 
Lattices $\Gamma_1$ and $\Gamma_2$ may not coincide, but they are both sub-lattices of $\mathfrak{Q} = \mathbf{Q}\cap \bbZ^5$.\\

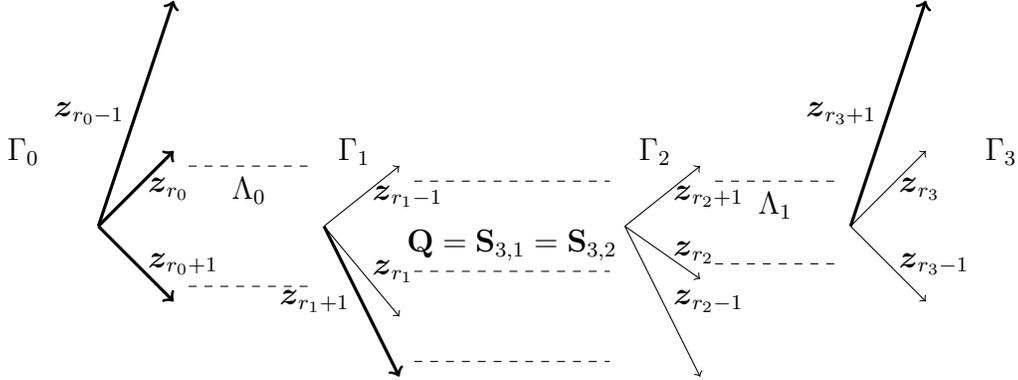
\begin{figure}[!h] 
 \begin{center}
\begin{tikzpicture}

\draw (-1, 1) node {$\LARGE{\Gamma_0}$};
\draw[very thick, ->] (0,0)--(1,3) node[midway,left] {$\bz_{r_0-1}$};
\draw[very thick, ->] (0,0)--(1,1) node[midway,right] {$\bz_{r_0}$};
\draw[very thick, ->] (0,0)--(1,-1) node[midway,right] {$\bz_{r_0+1}$};

\draw[dashed] (1.2,0.8)--(2.8,0.8) node[midway,below] {$\LARGE{\Lambda_0}$};
\draw[dashed] (1.2,-0.8)--(2.8,-0.8);

\draw[thin, ->] (3,0)--(4,0.8) node[midway,right] {$\bz_{r_1-1}$};
\draw[thin, ->] (3,0)--(4,-1.2) node[midway,right] {$\bz_{r_1}$};
\draw[very thick, ->] (3,0)--(4,-2) ;
\draw (3.1,-1.5) node {$\bz_{r_1+1}$};

\draw[dashed] (4.2,0.6)--(6.8,0.6);
\draw[dashed] (4.2,-0.6)--(6.8,-0.6) node [midway,above] {$\LARGE{\mathbf{Q} = \mathbf{S}_{3,1} =\mathbf{S}_{3,2}}$};
\draw[dashed] (4.2,-1.8)--(6.8,-1.8);

\draw[thin, ->] (7,0)--(8,0.8) node[midway,right] {$\bz_{r_2+1}$};
\draw[thin, ->] (7,0)--(8,-0.7) node[midway,right] {$\bz_{r_2}$};
\draw[thin, ->] (7,0)--(8,-2) ;
\draw (8.3,-1.2) node {$\bz_{r_2-1}$};

\draw[dashed] (8.2,0.6)--(9.8,0.6) node[midway,below] {$\LARGE{\Lambda_1}$};
\draw[dashed] (8.2,-0.5)--(9.8,-0.5);

\draw[very thick, ->] (10,0)--(11,3) node[midway,left] {$\bz_{r_3+1}$};
\draw[thin, ->] (10,0)--(11,1) node[midway,right] {$\bz_{r_3}$};
\draw[thin, ->] (10,0)--(11,-1) node[midway,right] {$\bz_{r_3-1}$};

\draw (12,1) node {$\LARGE{\Gamma_3}$};
\draw (7.4,1) node {$\LARGE{\Gamma_2}$};
\draw (3.4,1) node {$\LARGE{\Gamma_1}$};

 \end{tikzpicture}
 \end{center}
 \caption{Selected sequence of best approximation vectors.}\label{fig4nb}
 \end{figure}
 
 In Figure \ref{fig4nb}, the dashed lines should be interpreted as follows. The best approximation vectors $(\bz_l)_{r_0 \leq l \leq r_1}$ generate the $2$-dimensional lattice $\Lambda_0$. The best approximation vectors $(\bz_l)_{r_2 \leq l \leq r_3}$ generate the $2$-dimensional lattice $\Lambda_1$. The best approximation vectors $(\bz_l)_{r_1-1 \leq l \leq r_2+1}$ generate the $3$-dimensional rational subspace $\mathbf{Q} = \mathbf{S}_{3,1} =\mathbf{S}_{3,2}$. The five bold vectors span the whole space $\bbR^5$.\\

\section{Arbitrary dimension} \label{anynb}

\subsection{Two lemmas}

Consider $\btheta\in\bbR^n$ with $\bbQ$-linearly independent coordinates with $1$. Consider $(\bz_l)_{l\in\bbN}$ a sequence of best approximation vectors to $\btheta$.

\begin{lem}\label{lemnnb}
Let $k_0$ be an arbitrarily large index. There exists $2^{n-2}$ indices $ k_0 < r_0 < r_1 , \ldots , r_{2^{n-2}-2} < r_{2^{n-2}-1} $ such that the following holds.

\begin{enumerate}
\item{ The triples of consecutive best approximation vectors}
\[ \mathcal{S}_{3,l} = \{ \bz_{r_l-1}, \bz_{r_l}, \bz_{r_l+1} \}, \; \; \; 0 \leq l \leq 2^{n-2}-1 \]
consist of linearly independent vectors spanning a $3$-dimensional rational subspace $S_{3,l}$.
\item{For $4 \leq k \leq n+1$} and $ 0\leq l \leq 2^{n-k+1} -1$ , denote by $\mathcal{S}_{k,l}$ the set of best approximation vectors
\[ \mathcal{S}_{k,l} = \cup_{\nu=0}^{2^{k-3}-1}  \mathcal{S}_{3,2^{k-3}l + \nu} . \]
$\mathcal{S}_{k,l}$ spans the $k$-dimensional rational subspace $\mathbf{S}_{k,l}$.
\item{The rational subspaces $\mathbf{S}_{k,l}$} satisfy the relations
\begin{eqnarray}
\mathbf{S}_{k,2l} \cup \mathbf{S}_{k, 2l+1} &=& \mathbf{S}_{k+1,l},\\ \label{defQ}
\mathbf{S}_{k,2l} \cap \mathbf{S}_{k,2l+1} &=& \mathbf{S}_{k-1,4l+1} = \mathbf{S}_{k-1,4l+2} = : \mathbf{Q}_{k-1, l}.
\end{eqnarray}
In particular, $\mathbf{Q}_{2,l}$ is spanned by both $\bz_{r_{2l}}, \bz_{r_{2l}+1} $ and $\bz_{r_{2l+1}-1}, \bz_{r_{2l+1}}$.

\item{The full space $\bbR^{n+1}$ is spanned by}
\[\langle \bz_{r_0-1}, \bz_{r_0}, \bz_{r_0+1}, \bz_{r_1+1}, \bz_{r_2+1}, \ldots , \bz_{r_{2^{n-3}-1}+1} \rangle_\bbR = \langle \cup_{l=0}^{2^{n-k+1}-1} \mathbf{S}_{k,l} \rangle_\bbR, \; \; 3\leq k \leq n +1.\]
In particular, $\mathbf{S}_{n+1,0}= \bbR^{n+1}$.
\end{enumerate}
\end{lem}

Here, the first index always denotes the dimension of the considered object. For a given dimension $k$, there are $2^{n-k+1}$ subspaces $\mathbf{S}_{k,l}$ and $2^{n-k-1}$ subspaces $\mathbf{Q}_{k,l}$ of dimension $k$.\\

Another important pattern of best approximation vectors which may be useful for the considered problem was already discovered for any dimension in 2013 by V. Nguyen in \cite[\S2.3]{Ng} while studying simultaneous approximation to the basis of an algebraic number field and an extra real number.\\

Lemma \ref{lemnnb} coincide with Lemma \ref{Lem3nb} for the approximation to three numbers and with Lemma \ref{lem4nb} for the approximation to four numbers. In the later case, we have $ \Lambda_j \sim \mathfrak{Q}_{2,j}$ for $0 \leq j \leq 1$.\\

We can partially describe the situation with the binary tree from Figure \ref{fignnb}, where each child is included in its parent. In particular, the parent of a given rational subspace $\mathbf{S}_{k,l}$ is $\mathbf{S}_{k+1,\sigma(l)}$ where $\sigma$ is the usual shift on the binary expansion. \\

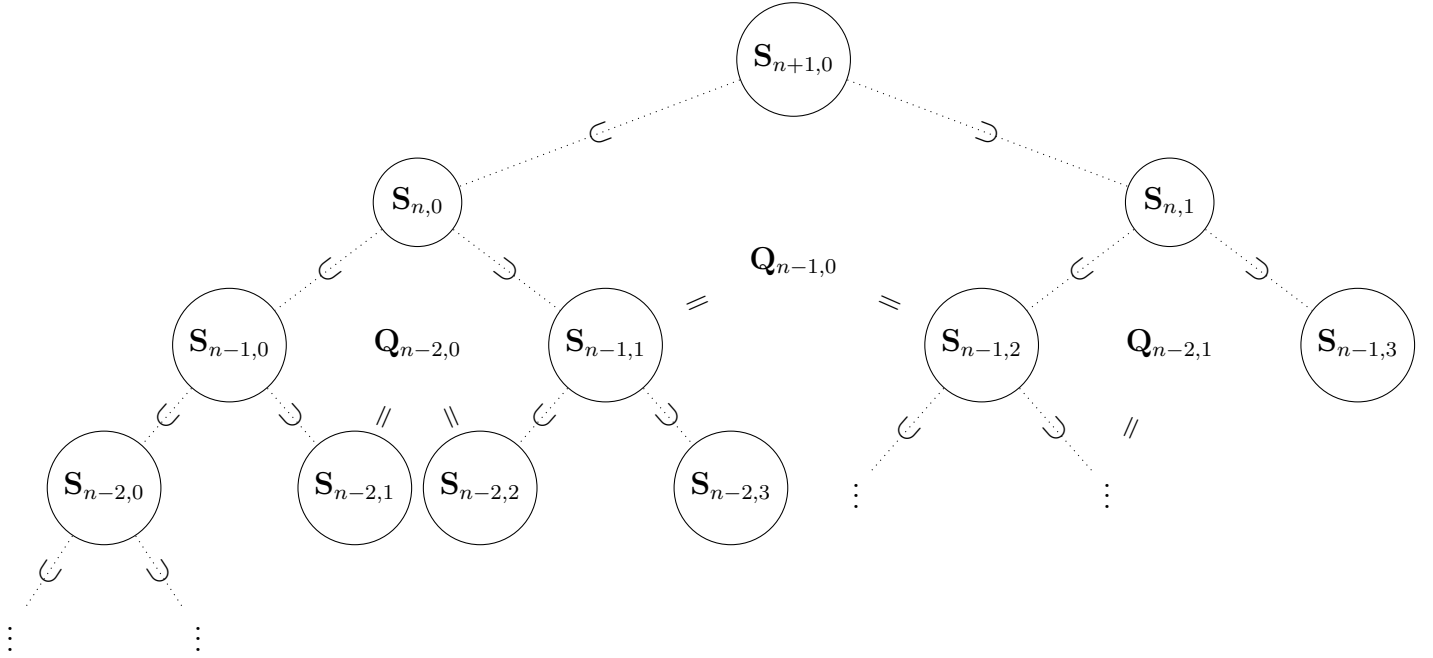
\begin{figure}[!h] 
 \begin{center}
\begin{tikzpicture}[level/.style={sibling distance=100mm/#1, level distance=19mm}]
\node [circle,draw] {${\mathbf{S}}_{n+1,0}$}
  child {node [circle,draw] (a) {${\mathbf{S}}_{n,0}$} 
    child {node [circle,draw,black,solid]  (c) {${\mathbf{S}}_{n-1,0}$}
        child {node [circle,draw,black,solid] (g)  {${\mathbf{S}}_{n-2,0}$}
         child {node  {$\vdots$}edge from parent  [thin,dotted,black] node[black,sloped] {$\subset$}}
    child {node {$\vdots$}edge from parent  [thin,dotted,black] node[black,sloped] {$\supset$}
    } edge from parent  [thin,dotted,black] node[black,sloped] {$\subset$}
    }
        child {node [circle,draw,black,solid] (j) {${\mathbf{S}}_{n-2,1}$} edge from parent  [thin,dotted,black] node[black,sloped] {$\supset$}
    } edge from parent  [thin,dotted,black] node[black,sloped] {$\subset$}
    }
    child {node [circle,draw,black,solid] (d) {${\mathbf{S}}_{n-1,1}$} 
     child {node [circle,draw,black,solid] (k) {${\mathbf{S}}_{n-2,2}$} edge from parent [thin,dotted,black] node [black,sloped] {$\subset$}}
    child {node [circle,draw,black,solid] {${\mathbf{S}}_{n-2,3}$}
    edge from parent  [thin,dotted,black]node[black,sloped] {$\supset$}
    } edge from parent [thin,dotted,black] node[black,sloped] {$\supset$}
  }
  edge from parent [thin,dotted,black] node[black,sloped] {$\subset$}}
  child {node [circle,draw,black,solid] (b) {${\mathbf{S}}_{n,1}$}
    child {node [circle,draw,black,solid] (e) {${\mathbf{S}}_{n-1,2}$} 
     child {node  {$\vdots$} edge from parent [thin,dotted,black] node[black,sloped] {$\subset$}}
    child {node (h) {$\vdots$} edge from parent [thin,dotted,black] node[black,sloped] {$\supset$}
    }edge from parent  node[sloped] {$\subset$}
    }
  child {node [circle,draw,black,solid] (f) {${\mathbf{S}}_{n-1,3}$}
 edge from parent [thin,dotted,black] node[black,sloped] {$\supset$}    }
   edge from parent [thin,dotted,black] node[black,sloped] {$\supset$} };
\path (a) -- (b) node (u) [circle,midway,below] {$\mathbf{Q}_{n-1,0}$};
\path (c) -- (d) node (v) [circle,midway] {$\mathbf{Q}_{n-2,0}$};
\path (e) -- (f) node (w) [circle, midway] {$\mathbf{Q}_{n-2,1}$};
\path (d) -- (u) node [midway, sloped] { $=$};
\path (e) -- (u) node [midway, sloped] { $=$};
\path (j) -- (v) node [midway, sloped] { $=$};
\path (k) -- (v) node [midway, sloped] { $=$};
\path (h) -- (w) node [midway, sloped] { $=$};

 \end{tikzpicture}
 \end{center}
 \caption{Binary tree sketching the situation described in Lemma \ref{lemnnb}. 
 }\label{fignnb}
 \end{figure}

We may write the recursive step of the construction of patterns as follows:
\[ \mathcal{S}_{n+1,0}  \quad : \quad \mathcal{S}_{n,0}\underset{{\mathbf{Q}}_{n-1,0}}{-}\mathcal{S}_{n,1}\]
where ${\mathbf{Q}}_{n-1,0}$ is a $n-1$ dimensional subspace. For the rational subspaces $\mathbf{S}_{n,0}$, $\mathbf{S}_{n,1}$ and ${\mathbf{Q}}_{n-1,0}$ and their underlying lattices $\mathfrak{S}_{n,0}$, $\mathfrak{S}_{n,1}$ and ${\mathfrak{Q}}_{n-1,0}$, Schmidt's inequality \eqref{SH} provides 
\begin{equation}\label{shift}
\frac{\det \mathfrak{S}_{n,0} \cdot \det \mathfrak{S}_{n,1}}{ \det \mathfrak{Q}_{n-1,0}} \gg 1.
\end{equation}
This relation enables us to shift the optimization equation in the next dimension as obtained in the next lemma. 

\begin{lem}\label{optiformula}
Let $n\geq4$. Consider the pattern of best approximation vectors $\mathcal{S}_{n+1,0}$ and its sub-patterns given by Lemma \ref{lemnnb}. Here as before, $\mathfrak{S}_{k,l} = \mathbf{S}_{k,l}\cap\bbZ^{n+1}$ and $\mathfrak{Q}_{k,l} = \mathbf{Q}_{k,l}\cap\bbZ^{n+1}$ are the integer points lattices of the rational subspace $\mathbf{S}_{k,l}$ and $\mathbf{Q}_{k,l}$. Then
\begin{equation}\label{MainEq}
\prod_{l=0}^{2^{n-4} -1} \left(\cfrac{\det\left(  \mathfrak{S}_{3,4l}\right) \det\left(  \mathfrak{Q}_{3,l} \right)^{1-y_{n-4}} }{ \det\left(  \mathfrak{Q}_{2,2l} \right)}\right)^{w_{n-4,l}} 
\left(\cfrac{ \det\left(  \mathfrak{Q}_{3,l} \right) ^{1-z_{n-4}}\det\left(  \mathfrak{S}_{3,4l+3} \right) }{ \det\left(  \mathfrak{Q}_{2,2l+1} \right)}\right)^{w_{n-4,l}'}  \gg 1,
\end{equation}
where the parameters $w_{k,l}, w_{k,l}', y_k$ and $z_k$ are defined inductively as follows.\\

Parameters
$y_0,z_0\in(0,1)$ are arbitrary such that 
\begin{equation}\label{inityz}
0 = y_{0} + z_0 -1
\end{equation}
 and then
\begin{eqnarray}\label{recyz}
(y_{k+1},z_{k+1}) &=& F(y_k,z_k) = \left( \cfrac{y_k}{y_k +z_k - y_k z_k}, \cfrac{z_k}{y_k + z_k - y_k z_k} \right),\\
1&=& w_{ 0,0} = w_{0,0}',
\end{eqnarray}
\begin{equation}\label{recw} w_{k+1,2l}=w_{k,l}, \; \; w_{k+1,2l+1} = (1-z_k)w_{k,l}', \; \; w_{k+1,2l}' = (1-y_k)w_{k,l} \; \; \textrm{ and } \; \; w_{k+1, 2l+1}' = w_{k,l}'.\end{equation}

Furthermore, the parameters satisfy the relations
\begin{eqnarray} \label{sumexpl1}
\sum_{l=0}^{2^{n-4} -1} (2-y_{n-4}) w_{n-4,l} + (2-z_{n-4}){w_{n-4,l}'} &= n-1,\\  \label{sumexpl2}
\sum_{l=0}^{2^{n-4} -1} w_{n-4,l} +{w_{n-4,l}'} &=n-2.
\end{eqnarray}
\end{lem}

We prove Lemma \ref{lemnnb} in subsection \ref{pf7} and then Lemma \ref{optiformula} in subsection \ref{pf8}. We first finish the proof of Theorem \ref{MainThm} in the case of simultaneous approximation.\\


\subsection{Proof of Theorem \ref{MainThm} }\label{PfMthm}

Consider $\btheta\in\bbR^n$ with $\bbQ$-linearly independent coordinates with $1$, and take $\alpha<\hat{\lambda}(\btheta)$. Denote by $g$ the unique positive root of $R_{n,\alpha}$ defined by \eqref{poly}.\\

Choose $y_0$ and $z_0$ in the following way:
\begin{equation}
y_0= \frac{R_{n-1,\alpha}(g)}{gR_{n-2,\alpha}(g)}, z_0 =  \frac{R_{n-1,\alpha}(g)}{R_{n-2,\alpha}(g)}.
\end{equation}
Using \eqref{aaa}, one can check the condition
\[y_{0} + z_{0} -1 = \frac{R_{n-1,\alpha}(g) +gR_{n-1,\alpha}(g) -gR_{n-2,\alpha}(g)}{gR_{n-2,\alpha}(g)} = \frac{R_{n,\alpha}(g)}{gR_{n-2,\alpha}(g)}=0.\]
By the induction formula \eqref{recyz}, we deduce that for every $4\leq k \leq n$
\begin{equation}\label{toprove}
y_{n-k} = \frac{R_{k-1,\alpha}(g)}{gR_{k-2,\alpha}(g)} \; \; \textrm{ and }\; \; z_{n-k} = \frac{R_{k-1,\alpha}(g)}{R_{k-2,\alpha}(g)}.\end{equation}
Indeed, the formula  (\ref{toprove}) is satisfied for $n-k=0$. Suppose that it is valid for a certain value of $k$. Then 
$\frac{z_{n-k}}{y_{n-k}} = g$ and recursive formula  (\ref{recyz}) gives us
 $\frac{z_{n-k-1}}{y_{n-k-1}} = g$. It is enough for verifying (\ref{recyz})  with $k$ replaced by $k+1$ by means of the first
 group of equalities from
 \eqref{aaa}.


In particular, we have 
\[ y_{n-4} = \frac{R_{3,\alpha}(g)}{gR_{2,\alpha}(g)} \; \; \textrm{ and }\; \; z_{n-4} = \frac{R_{3,\alpha}(g)}{R_{2,\alpha}(g)}.\]
We consider $g(s,t)$ defined in \eqref{defg} (Lemma \ref{Keylemma}) for the parameters
\begin{equation}\label{st1}
s =1, \; t= 1-y_{n-4}
\end{equation}
and
\begin{equation}\label{st2}
s=1-z_{n-4}, \;t=1.
\end{equation}

From \eqref{guy2} and \eqref{guy3}, it follows that

\begin{equation}\label{equalg}
g = G(n,\alpha) = g(1,1-z_{n-4}) = g(1-y_{n-4},1).\end{equation}


Recall that now $ g = G(n,\alpha)$ is the root of the polynomial $ R_{n,\alpha}(x)$.
 So the positivity condition \eqref{positivity}   for parameters (\ref{st1}) follows from \eqref{1pos-special-case}.
 At the same time for parameters 
 (\ref{st2})
 the positivity condition \eqref{positivity}  is clearly true. \\

According to Lemma \ref{optiformula}, we have \eqref{MainEq} and therefore there exists an index $0 \leq l \leq 2^{n-4}-1$ such that either
\[ \cfrac{\det\left(  \mathfrak{S}_{3,4l}\right) \det\left(  \mathfrak{Q}_{3,l} \right)^{1-y_{n-4}} }{ \det\left(  \mathfrak{Q}_{2,2l} \right)}  \gg 1 \; \; \textrm{ or } \; \;  \cfrac{ \det\left(  \mathfrak{Q}_{3,l} \right) ^{1-z_{n-4}}\det\left(  \mathfrak{S}_{3,4l+3} \right) }{ \det\left(  \mathfrak{Q}_{2,2l+1} \right)}  \gg 1.\]

To summarize, all the conditions are met to apply Lemma \ref{Keylemma} for either 
\[\mathfrak{S}_- = \mathfrak{S}_{3,4l}, \; \mathfrak{S}_+ =\mathfrak{Q}_{3,l}, \;   s =1, \; t= 1-y_{n-4}\]
or
\[\mathfrak{S}_- =  \mathfrak{Q}_{3,l}, \;\mathfrak{S}_+ = \mathfrak{S}_{3,4l+3}, \;  s=1-z_{n-4}, \;t=1.\]
Hence, there exists $\nu$ with $q_{\nu+1}\gg q_\nu^{g}$ and \eqref{equivthm} is met, proving Theorem \ref{MainThm}. \qed

\subsection{Proof of Lemma \ref{lemnnb}}\label{pf7}

Figure \ref{fignnb} may be useful to understand the construction.\\
Let $k_0\gg1$. We prove the lemma by induction in the dimension $n$. Suppose that we can construct a pattern $\mathcal{S}_{m,0}$ of $2^{m-3}$ triples of consecutive best approximation vectors given by indices $ k_0 < r_0 < r_1 , \ldots , r_{2^{m-3}-2} < r_{2^{m-3}-1} $ spanning a $m$-dimensional rational space. Such a construction for $m=4,5$ holds via Lemmas \ref{Lem3nb} and \ref{lem4nb}. This provides the base of induction.\\

Consider $(\bz_l)_{l\in\bbN}$ a sequence of best approximation vectors spanning a $(m+1)$-dimensional rational space $\mathbf{S}_{m+1}$.
Set $r_{2^{m-2}-1}$ to be the smallest index such that 
\[ \langle \bz_{k_0}, \bz_{k_0+1}, \ldots , \bz_{r_{2^{m-2}-1}}, \bz_{r_{2^{m-2}-1}+1}\rangle_\bbR = \mathbf{S}_{m+1}.\]
Note that $\bz_{r_{2^{m-2}-1}+1}$ is not in the $m$-dimensional subspace spanned by $\left(\bz_l\right)_{k_0 \leq l \leq r_{2^{m-2}}}$. In particular, since two consecutive best approximation vectors are linearly independent the three consecutive best approximation vectors $\bz_{r_{2^{m-2}-1}-1}, \bz_{r_{2^{m-2}-1}}, \bz_{r_{2^{m-2}-1}+1}$ are linearly independent and span a $3$-dimensional subspace denoted by $\mathbf{S}_{3, 2^{m-2}-1}$. Set $r_0 > k_0$ to be the largest index such that 
\[ \langle \bz_{r_0-1}, \bz_{r_0}, \ldots , \bz_{r_{2^{m-2}-1}}, \bz_{r_{2^{m-2}-1}+1} \rangle_\bbR = \mathbf{S}_{m+1}.\] 
Note that $\bz_{r_0-1}$ is not in the $m$-dimensional subspace spanned by $\left(\bz_l\right)_{r_0 \leq l \leq r_{2^{m-1}-1} +1}$. In particular, since two consecutive best approximation vectors are linearly independent the three consecutive best approximation vectors $\bz_{r_0-1}, \bz_{r_0}, \bz_{r_0+1}$ are linearly independent and span a $3$-dimensional subspace denoted by $\mathbf{S}_{3,0}$. Moreover, combining both observations we get that 
\[ \mathbf{Q}_{m-1,0} : = \langle \bz_{r_0}, \bz_{r_0+1}, \ldots , \bz_{r_{2^{m-2}-1}-1}, \bz_{r_{2^{m-2}-1}}\rangle_\bbR \] 
is a $m-1$-dimensional subspace.\\
We use the induction hypothesis for the two $m$-dimensional subspaces 
\[ \mathbf{S}_{m}' :=\langle \bz_{r_0-1}, \bz_{r_0}, \ldots , \bz_{r_{2^{m-2}-1}-1}, \bz_{r_{2^{m-2}-1}}\rangle_\bbR \; \; \textrm{ and } \; \;  \mathbf{S}_{m}'' :=   \langle \bz_{r_0}, \bz_{r_0+1}, \ldots , \bz_{r_{2^{m-2}-1}}, \bz_{r_{2^{m-2}-1}+1} \rangle_\bbR \]
for $k_0'=r_0-1$ and $k_0''=r_0$ respectively. This provides two patterns $\mathcal{S}_{m}'$ and $\mathcal{S}_{m}''$ of triples of best approximation vectors defined by indices $r_0 \leq r_0' < r_1' , \ldots , r_{2^{m-3}-2}' < r_{2^{m-3}-1}'$ and $r_0+1 \leq r_0'' < r_1'', \ldots , r_{2^{m-3}-2}'' < r_{2^{m-3}-1}''$ satisfying the conditions of Lemma \ref{lemnnb}. A key observation is that by definition of $r_0$, we necessarily have $r_0'=r_0$. Similarly, by definition of $r_{2^{m-2}-1}$, we necessarily have $r_{2^{m-2}-1} = r_{2^{m-3}-1}''$. It follows that both sub-patterns $\mathcal{S}_{m-1,1}'$ and $\mathcal{S}_{m-1,0}''$ span the rational subspace $\mathbf{Q}_{m-1,0}$. Hence, the pattern $\mathcal{S}$ defined by the triples given by indices 
\[ r_i = r_i' \; \; \textrm{ and } \; \; r_{i + 2^{m-3}} = r_i'' \; \; \textrm{ for } 0\leq i \leq 2^{m-3}-1 \]
combining the two sub-patterns $\mathcal{S}_{m}'$ and $\mathcal{S}_{m}''$ satisfies the required properties at the rank $m+1$. \[\mathcal{S}  \quad : \quad \mathcal{S}_{m}' \underset{\mathbf{Q}_{m-1,0}}{-} \mathcal{S}_{m}''.\] 

Since  $(\bz_l)_{l\in\bbN}$ a sequence of best simultaneous approximation vectors to $\btheta\in\bbR^n$ spans the whole space $\bbR^{n+1}$, Lemma \ref{lemnnb} follows.\qed

\begin{rmq}\label{rmk}
Note that the proof provides a $m$-dimensional pattern for $\btheta\in\bbR^n$ where $m$ is the dimension of the space spanned by its best approximation vectors. Furthermore, note that this construction holds for both simultaneous approximation and approximation by one linear form.\\
\end{rmq}

\subsection{Proof of Lemma \ref{optiformula}}\label{pf8}
By induction on $k$ we prove a more general formula
\begin{equation}\label{formrec}
\begin{split}
\prod_{l=0}^{2^{k-1} -1} \left(\cfrac{\det\left( \mathfrak{S}_{n-k,4l}\right) \det\left(  \mathfrak{S}_{n-k,4l+1} \right)^{1-y_{k-1}} }{ \det\left(  \mathfrak{S}_{n-k-1,8l+1} \right)}\right)^{w_{k-1,l}}  \times \quad \quad & \\
 \quad \prod_{l=0}^{2^{k-1} -1} \left(\cfrac{ \det\left( \mathfrak{S}_{n-k,4l+2} \right) ^{1-z_{k-1}}\det\left( \mathfrak{S}_{n-k,4l+3} \right) }{ \det\left(  \mathfrak{S}_{n-k-1,8l+3} \right)}\right)^{w_{k-1,l}'}  & \gg 1.
\end{split}
\end{equation}
 
 If we write it in terms of $\mathfrak{Q}_{i,j} = \mathfrak{S}_{i,4j+1} = \mathfrak{S}_{i,4j+2}$, we have

\begin{equation}
\begin{split}
\prod_{l=0}^{2^{k-1} -1} \left(\cfrac{\det\left( \mathfrak{S}_{n-k,4l}\right) \det\left(  \mathfrak{Q}_{n-k,l} \right)^{1-y_{k-1}} }{ \det\left(  \mathfrak{Q}_{n-k-1,2l} \right)}\right)^{w_{k-1,l}}  \times \quad \quad & \\
 \quad \prod_{l=0}^{2^{k-1} -1} \left(\cfrac{ \det\left( \mathfrak{Q}_{n-k,l} \right) ^{1-z_{k-1}}\det\left( \mathfrak{S}_{n-k,4l+3} \right) }{ \det\left(  \mathfrak{Q}_{n-k-1,2l+1} \right)}\right)^{w_{k-1,l}'}  & \gg 1.
\end{split}
\end{equation}

Lemma \ref{optiformula} is the latter formula for $k=n-3$. \\

We call factors of the first product, of the form \[\left(\cfrac{\det\left( \mathfrak{S}_{n-k,4l}\right) \det\left(  \mathfrak{Q}_{n-k,l} \right)^{1-y_{k-1}} }{ \det\left(  \mathfrak{Q}_{n-k-1,2l} \right)}\right)^{w_{k-1,l}} \] factors of \emph{Type I}, and factors of the second product of the form \[\left(\cfrac{ \det\left( \mathfrak{Q}_{n-k,l} \right) ^{1-z_{k-1}}\det\left( \mathfrak{S}_{n-k,4l+3} \right) }{ \det\left(  \mathfrak{Q}_{n-k-1,2l+1} \right)}\right)^{w_{k-1,l}'} \] factors of \emph{Type II}.\\

\paragraph{Base of induction} follows the steps of approximation to four numbers. Namely, Schmidt's inequality \eqref{SH} provides
\begin{equation}
\begin{cases}
\det (\mathfrak{S}_{n,0}) \det (\mathfrak{S}_{n,1}) &\gg \det( \mathfrak{Q}_{n-1,0} )\det (\mathfrak{S}_{n+1,0}) \; ,\\
\det (\mathfrak{S}_{n-1,0}) \det (\mathfrak{S}_{n-1,1}) &\gg \det( \mathfrak{Q}_{n-2,0}) \det( \mathfrak{S}_{n,0}) \; , \\
\det (\mathfrak{S}_{n-1,2}) \det (\mathfrak{S}_{n-1,3}) &\gg \det (\mathfrak{Q}_{n-2,1} )\det (\mathfrak{S}_{n,1}) \; .
\end{cases}
\end{equation}

Since $\mathcal{S}_{n+1,0}$ spans the whole space $\bbR^{n+1}$, we have $\det \mathfrak{S}_{n+1,0} =1$ and using the fact that $\det \mathfrak{Q}_{n-1,0}= \det \mathfrak{S}_{n-1,1}= \det \mathfrak{S}_{n-1,2}$ (by \eqref{defQ} ), we get the formula
\[ \cfrac{\det( \mathfrak{S}_{n-1,0}) \det( \mathfrak{Q}_{n-1,0}) \det( \mathfrak{S}_{n-1,3})}{\det( \mathfrak{Q}_{n-2,0}) \det( \mathfrak{Q}_{n-2,1})} \gg 1. \]
Setting $w_{0,0}=w_{0,0}'=1$ and $y_0$ and $z_0$ such that $y_0+z_0-1=0$, we can rewrite
\[ \left(\cfrac{\det( \mathfrak{S}_{n-1,0}) \det( \mathfrak{Q}_{n-1,0})^{1-y_0}}{\det (\mathfrak{Q}_{n-2,0})}\right)^{w_{0,0}} \left(\cfrac{ \det( \mathfrak{Q}_{n-1,0})^{1-z_0} \det( \mathfrak{S}_{n-1,3})}{ \det (\mathfrak{Q}_{n-2,1})}\right)^{w_{0,0}'} \gg 1. \]
This establishes the expected formula for $k=1$. In the inductive step, Schmidt's inequality \eqref{SH} splits each term of the product in two terms involving rational subspaces of lower dimension, and shift the values of the parameters $y_k$ and $z_k$.\\

Indeed, for $3\leq i \leq n+1 $ and $0 \leq j \leq 2^{n+1-i} -1$, Schmidt's inequality provides

\begin{equation}\label{shgen}
\cfrac{\det( \mathfrak{S}_{i-1,2j}) \det( \mathfrak{S}_{i-1,2j+1})}{\det ( \mathfrak{Q}_{i-2,j})} \gg \det (\mathfrak{S}_{i,j}).
\end{equation}

\begin{figure}[!h] 
 \begin{center}
\begin{tikzpicture}[level/.style={sibling distance=140mm/#1}]
\node  {$\vdots$}
child { node  (aa) {$\mathbf{S}_{i,j}$} edge from parent[draw=none]
  child {node  (a) {$\mathbf{S}_{i-1,2j}$}
    child {node (c) {$\mathbf{S}_{i-2,4j}$} edge from parent [thin,dotted,black] node[sloped] {$\subset$}}
    child {node  (d) {$\mathbf{S}_{i-2,4j+1}$} edge from parent [thin,dotted,black] node[sloped] {$\supset$}
    }  edge from parent [thin,dotted,black] node[sloped] {$\subset$}}
  child {node (b) {$\mathbf{S}_{i-1,2j+1}$}
    child {node  (e) {$\mathbf{S}_{i-2,4j+2}$} edge from parent [thin,dotted,black] node[sloped] {$\subset$}
    }
  child {node  (f) {$\mathbf{S}_{i-2,4j+3}$} edge from parent [thin,dotted,black] node[sloped] {$\supset$}
    }  edge from parent [thin,dotted,black] node[sloped] {$\supset$}}
    };
\path (a) -- (b) node (u) [midway] {$\mathbf{Q}_{i-2,j}$};
\path (d) -- (u) node [midway, sloped] { \Large{$=$}};
\path (e) -- (u) node [midway, sloped] { \Large{$=$}};

\path   (e) -- (f) node [midway] { \tiny{Type II} };
\path  (c) -- (d)  node [midway] { \tiny{Type I} };

 \end{tikzpicture}
 \end{center}
 \caption{Situation to apply Schmidt's inequality.}\label{fignnb2}
 \end{figure}
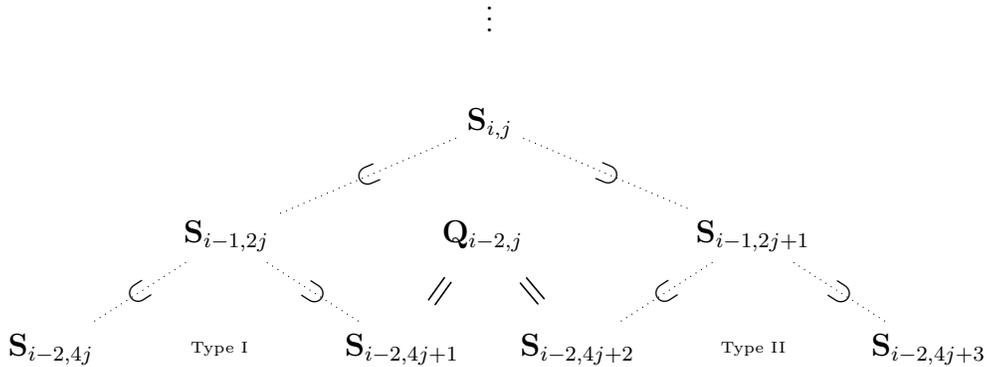

\paragraph{Inductive step.} Assume that \eqref{formrec} holds for some $1 \leq k < n-3$. In the product \eqref{formrec}, there are two types of factors: factors of Type I and of Type II. Each of these factors splits into two factors, one of Type I and one of Type II. We first deal with factors of Type I. For every $0\leq l \leq 2^{k-1}-1$,  we can apply Schmidt's inequality \eqref{shgen} with parameters $i=n-k$ and $j=4l$ and $j=4l+1$ respectively to split
\begin{eqnarray}\label{observation}
& & \left(\cfrac{\det\left(  \mathfrak{S}_{n-k,4l}\right) \det\left(  \mathfrak{S}_{n-k,4l+1} \right)^{1-y_{k}} }{ \det\left(  \mathfrak{Q}_{n-k-1,2l} \right)}\right)^{w_{k,l}}\\ \notag
  & & \; \ll \left(       \cfrac{\left(  \cfrac{ \det(\mathfrak{S}_{n-k-1,8l} ) \det( \mathfrak{S}_{n-k-1,8l+1})  }{\det(\mathfrak{Q}_{n-k-2,4l})}     \right) \left(     \cfrac{\det( \mathfrak{S}_{n-k-1,8l+2}) \det(\mathfrak{S}_{n-k-1,8l+3})}{\det(\mathfrak{Q}_{n-k-2,4l+1})}        \right)^{1-y_{k}} }{ \det\left( \mathfrak{Q}_{n-k-1,2l} \right)}        \right)^{w_{k,l}}.
\end{eqnarray}
Considering that $\mathbf{Q}_{n-k-1,2l} = \mathbf{S}_{n-k-1,8l+1} = \mathbf{S}_{n-k-1,8l+2}$, for any $u \in (0,1)$ we can write
\begin{equation}\label{part1}
\begin{split}
 \left(    \cfrac{ \det(\mathfrak{S}_{n-k-1,8l} ) \det( \mathfrak{S}_{n-k-1,8l+1})^{u(1-y_k)}  }{\det(\mathfrak{Q}_{n-k-2,4l})}  \right)^{w_{k,l}} \times \quad \quad &\\
\quad \quad  \left(   \cfrac{\det(\mathfrak{S}_{n-k-1,8l+2})^{1-u} \det(\mathfrak{S}_{n-k-1,8l+3})}{\det(\mathfrak{Q}_{n-k-2,4l+1})}     \right)^{(1-y_k)w_{k,l}} &\gg \eqref{observation}.
 \end{split}
\end{equation}

Similarly, for factors of Type II, for any $v \in (0,1)$, using \eqref{shgen} with $i=n-k$ and $j=4l+2$ and $j=4l+3$ respectively, and the fact that $\mathbf{Q}_{n-k-1,2l+1} = \mathbf{S}_{n-k-1,8l+5} = \mathbf{S}_{n-k-1,8l+6}$ we get

\begin{equation}\label{part2}\begin{split}
\left(\cfrac{ \det\left( \mathfrak{S}_{n-k,4l+2} \right) ^{1-z_{k-1}}\det\left( \mathfrak{S}_{n-k,4l+3} \right) }{ \det\left(  \mathfrak{Q}_{n-k-1,2l+1} \right)}\right)^{w_{k,l}'}  \ll &\left(    \cfrac{ \det(\mathfrak{S}_{n-k-1,8l+4} ) \det( \mathfrak{S}_{n-k-1,8l+5})^{1-v}  }{\det(\mathfrak{Q}_{n-k-2,4l+2})}  \right)^{(1-z_k)w_{k,l}'}  \\
&\quad \quad  \times \left(   \cfrac{\det(\mathfrak{S}_{n-k-1,8l+6})^{v(1-z_{k})} \det(\mathfrak{S}_{n-k-1,8l+7})}{\det(\mathfrak{Q}_{n-k-2,4l+3})}     \right)^{w_{k,l}'}.
\end{split}\end{equation}

Combining the splitting of Type I \eqref{part1} and  Type II \eqref{part2} factors in the induction hypothesis \eqref{formrec}, it appears that we should define the parameters $(y_{k+1},z_{k+1})$ to be solutions of the system in variables $(v,u)$
\[ u(1-y_k) = 1- v \; \; \textrm{ and } \; \; 1-u=v(1-z_k). \]
That is
\[ y_k = \cfrac{y_{k+1}+z_{k+1}-1}{z_{k+1}} \; \; \textrm{ and } z_k = \cfrac{y_{k+1}+z_{k+1}-1}{y_{k+1}} \]
or equivalently 
\[ y_{k+1} = \cfrac{y_k}{y_k +z_k - y_k z_k}  \; \; \textrm{ and } z_{k+1} = \cfrac{z_k}{y_k + z_k - y_k z_k}. \]
The last equality coincide with the definition $F(y,z)$ in \eqref{recyz} and \eqref{part2} gives formulae \eqref{recw} for $w_{k+1,l}$ and $w'_{k+1,l}$.\\

This and the parameters \eqref{recw} establish formula \eqref{formrec} for $k+1$.\\

We now prove the relation \eqref{sumexpl1} and \eqref{sumexpl2} by descending induction, showing that for any $4 \leq k \leq n$
\begin{eqnarray}\label{recsum1}
\sum_{l=0}^{2^{n-k} -1}\left( (2-y_{n-k}) w_{n-k,l} + (2-z_{n-k}){w_{n-k,l}'} \right) &= n-k+3,\\ \label{recsum2}
\sum_{l=0}^{2^{n-k} -1} \left( w_{n-k,l} +{w_{n-k,l}'} \right) &=n-k+2.
\end{eqnarray}

First, note that 
\[ w_{0,0} + w_{0,0}' = 2, \; \textrm{ and } \; w_{0,0}(2-y_0) + w_{0,0}'(2-z_0) = 3,\]
hence we have the base of induction at $k=n$. \\

Assume that for some $4 \leq k \leq n$ \eqref{recsum1} and \eqref{recsum2} holds. The two sums represent the degrees of determinants that appears respectively at the numerator and at the denominator in \eqref{MainEq}. The key is to observe the splitting in \eqref{observation} : the new sum for the denominator is the sum from the previous numerator, while at the numerator, the previous denominator is doubled but we have a cancellation with one denominator. Namely, using the recurrence formula \eqref{recyz} and \eqref{recw} for the parameters, we get

\begin{eqnarray*}
\sum_{l=0}^{2^{n-k+1} -1} \left( w_{n-k+1,l} +{w_{n-k+1,l}'} \right)&=& \sum_{l=0}^{2^{n-k} -1}\left( w_{n-k+1,2l} +w_{n-k+1,2l+1} +{w_{n-k+1,2l}'} +{w_{n-k+1,2l+1}'}\right) \\
 &=& \sum_{l=0}^{2^{n-k+1} -1}\left( w_{n-k,l} +w'_{n-k,l}(1-z_{n-k}) +{w_{n-k,l}'} +{w_{n-k,l}}(1-y_{n-k})\right)\\
 &=& \sum_{l=0}^{2^{n-k} -1}\left( (2-y_{n-k}) w_{n-k,l} + (2-z_{n-k}){w_{n-k,l}'} \right)= n-k+3,
\end{eqnarray*}

\begin{eqnarray*}
\sum_{l=0}^{2^{n-k+1} -1} \left( w_{n-k+1,l}(2-y_{n-k+1}) +{w_{n-k+1,l}'} (2-z_{n-k+1})\right) &=& \\
= \sum_{l=0}^{2^{n-k} -1} (w_{n-k+1,2l}+ w_{n-k+1,2l+1})(2-y_{n-k+1}) & +& \sum_{l=0}^{2^{n-k} -1} ({w_{n-k+1,2l}'} +{w_{n-k+1,2l+1}'})(2-z_{n-k+1}) \\
 = \sum_{l=0}^{2^{n-k} -1} \left( w_{n-k,l} (3- 2y_{n-k}) +w'_{n-k,l} (3-2z_{n-k})\right) &=& 2(n-k+3) - (n-k+2) = n-k+4.
\end{eqnarray*}

Hence the result by descending induction.\qed

\section{Approximation by one linear form}\label{lf}

In this section, we explain how the very same geometry of a sequence of best approximation vectors provides Theorem \ref{MainThm} for approximation by one linear form. We need to consider a hyperbolic rotation to get a suitable analogue of the estimates in Lemma \ref{estim}. For this, we use Schmidt's inequalities on heights in a slightly larger context than rational subspaces.

\subsection{About Schmidt's inequalities on heights}\label{lfl8}

As stated in Proposition \ref{SchmidtHeight}, Schmidt's inequality deals with  the intersections of rational subspaces with the lattice $\mathbb{Z}^d$ of integer points. Here we need to deal with a more general situation. Let $\Lambda \subset \bbR^d$ be a complete lattice, that plays the role of integer points. Let $\bf{M}\subset \bbR^d$ be a subspace, it is called $\Lambda$-rational if the lattice $\mathfrak{M}= \textbf{M} \cap \Lambda$
is complete, i.e. if $\langle \mathfrak{M} \rangle_{\bbR} = \textbf{M} $.

\begin{lem}
The intersection of two $\Lambda$-rational subspaces is $\Lambda$-rational.
\end{lem}
The proof is the same as for rational subspaces, and use the description of subspaces by their orthogonal vectors.
\begin{defn}
Given a fixed complete lattice $\Lambda$, we define the height $H_{\Lambda}$ of a $\Lambda$-rational subspace $\textbf{M}$ to be the fundamental volume 
\[ H_{\Lambda}(\textbf{M}) = \det( \mathfrak{M} ) = \det(\textbf{M}\cap \Lambda) \]
of the $\Lambda$-points of $\textbf{M}$.
\end{defn}

\begin{prop}[Schmidt's inequality]\label{SchmidtHeightGen}
Let $\Lambda$ be a complete lattice. Let $\textbf{M}_1,\textbf{M}_2$ be two $\Lambda$-rational subspaces in $\bbR^d$, we have
\begin{equation}\label{SHlambda}
H_\Lambda(\textbf{M}_1+ \textbf{M}_2) \cdot H_\Lambda(\textbf{M}_1\cap \textbf{M}_2) \ll H_\Lambda(\textbf{M}_1)\cdot H_\Lambda(\textbf{M}_2).
\end{equation}
\end{prop}


\begin{proof}
 
 Let $\textbf{N} =  \textbf{M}_1 \cap \textbf{M}_2 $
and denote 
\[ {\rm dim}\, \textbf{M}_j = m_j,\,j=1,2,\,\,\,\,\,\,\,\, {\rm dim}\, \textbf{N} = n,\,\,\,\,  {\rm dim}\,( \textbf{M}_1+\textbf{M}_2) = f.
\]
Then
\[
f = m_1+m_2-n.
\]
Consider the orthogonal complement $\textbf{K}$ to $\textbf{N}$, ${\rm dim}\, \textbf{K}\cap \textbf{M}_j  = m_j-n.$
Let $ \mathcal{N} $ be a basis in $\textbf{N}$. For $j=1,2$, we take a collection of vectors $\mu_j\subset \textbf{M}_j$ in such a way that  the collection $\mathcal{M}_j = \mathcal{N}\cup \mu_j$ forms a basis of $\textbf{M}_j$. This means that we complete $\mathcal{N}$ by  $\mu_j$ to a basis of $\mathbf{M}_j$.
Let $\mu_j^*$ be a collection of independent vectors in $\textbf{K}\cap \textbf{M}_j$ which can me obtained  from $\mu_j$ by orthogonal projection  on $\textbf{K} $ parallel to 
$\textbf{N}$. Let us consider the parallelepiped $ \Pi \subset   \textbf{M}_1+\textbf{M}_2$ generated by all the vectors from the collection
 $\mathcal{N}\cup \mu_1\cup \mu_2$, and the parallelepiped $ \Pi ^*\subset   \textbf{M}_1+\textbf{M}_2$ generated by all the vectors from the collection
 $\mathcal{N}\cup \mu_1^*\cup \mu_2^*$. We consider also the parallelepipeds
 \[
 \Pi_\textbf{N},\,\,\,\, \Pi_{1}^*,\,\,\,\,\Pi_2^*
 \]
 generated by the collections of independent vectors
 \[ \mathcal{N},\,\,\,\, \mu_1^*,\,\,\,\,\mu_2^*
 \]
 correspondingly.  
 Also we need to consider parallelepipeds
 \[
  \Pi_{\textbf{M}_j}, \,\,\,\,j=1,2
 \]
 corresponding to the collections
 \[
 \mathcal{M}_j = \mathcal{N}\cup \mu_j,\,\,\,\, j=1,2.
 \]
 It is clear that 
  \[
  {\rm vol}_{m_j}\, \Pi_{\textbf{M}_j}  = {\rm vol}_{m_j -n}\, \Pi_j^* \cdot    {\rm vol}_{n}\, \Pi_{\textbf{N}},\,\,\,\,\,
   j =1,2
  \]
  and
  \[
    {\rm vol}_{f}\, \Pi =
    {\rm vol}_{f}\, \Pi^*
    \le  {\rm vol}_{m_1 -n}\, \Pi_1^* \cdot   {\rm vol}_{m_2-n}\, \Pi_2^*
    \cdot    {\rm vol}_{n}\, \Pi_{\textbf{N}} =
    \frac{{\rm vol}_{m_1 }\, \Pi_{\textbf{M}_1} \cdot   {\rm vol}_{m_2}\, \Pi_{\textbf{M}_2}}{{\rm vol}_{n}\, \Pi_{\textbf{N}} },
  \]
where ${\rm vol}_k (\cdot) $ stands for $k$-dimensional volume. So
  \[
    {\rm vol}_{f}\, \Pi  \cdot
      {{\rm vol}_{n}\, \Pi_{\textbf{N}} }
    \le
    {{\rm vol}_{m_1 }\, \Pi_{\textbf{M}_1} \cdot   {\rm vol}_{m_2}\, \Pi_{\textbf{M}_2}} ,
  \]
To obtain (\ref{SHlambda}) we need to apply the last inequality in the case when
$\mathcal{N}$ is a basis of the lattice $\Lambda\cap (\textbf{M}_1\cap \textbf{M}_2)$  and $ \mu_1, \mu_2$ complete $\mathcal{N}$ to the basises of lattices $\Lambda\cap \textbf{M}_1$
and $\Lambda\cap \textbf{M}_2$ correspondingly.
\end{proof}

\subsection{Hyperbolic rotation}
Given a sequence $(\bz_l)_{l\in\bbN} ={}^t(q_{1,l}, \ldots , q_{n,l}, a_l)$ of best approximations to a point $\btheta\in \bbR^n$ for the approximation by one linear form, we can extract a subsequence satisfying Lemma \ref{lemnnb}. For approximation by one linear form, it may happen that the sequence of best approximation vectors spans a subspace of dimension $m<n+1$ in $\bbR^{n+1}$ (see \cite{Chev}). In this case, Proposition \ref{p-decrG} provides that Theorem \ref{MainThm} holds with the stronger lower bound $G^*(m,\hat{\omega}(\btheta))$ instead of $G^*(n,\hat{\omega}(\btheta))$. See the remark after the proof of  Lemma \ref{lemnnb}.  In the sequel, we suppose that the best approximation vectors span the full space. In particular the coordinates $\theta_1, \ldots , \theta_n$ are linearly independent with $1$.\\
Consider the matrix $$
 {L} =
\left(
 \begin{array}{cccc}
 1& \cdots & 0&0\cr
 \vdots &\vdots& \vdots&\vdots\cr
 0&\cdots&1&0\cr
 \theta_1&\cdots &\theta_n&1
 \end{array}
 \right).
 $$

We can consider the sequence of best approximation as points of the lattice $\mathcal{L}= L.\bbZ^{n+1}$ with
\[ (\tilde{\bz}_l)_{l\in\bbN} = {L}. (\bz_l)_{l\in\bbN} \in \mathcal{L}.\]
Here, we simply replace the last coordinate $a_l$ by the error of approximation $L_l$.\\

Consider a large parameter $T$, and the hyperbolic rotation
$$
 {\cal G}_T =
 \left(
 \begin{array}{cccc}
 T^{-1}& \cdots & 0&0\cr
 \vdots &\vdots& \vdots&\vdots\cr
 0&\cdots&T^{-1}&0\cr
 0&\cdots &0&T^n
 \end{array}
 \right).
 $$

The lattice $\mathcal{L}'= {\cal G}_T {\cal  L}$ is complete since the determinants of $L$ and ${\cal G}_T$ are $1$.\\

Consider the sequence $(\bz'_l)_{l\in\bbN} = {\cal G}_T { L} (\bz_l)_{l\in\bbN} \in \mathcal{L}'$ where
\begin{eqnarray*}
\bz_l'= {}^t(z'_{1,l}, \ldots , z'_{n,l}, z'_{ n+1,l}) = {}^t(T^{-1}q_{1,l}, \ldots , T^{-1}q_{n,l}, T^nL_l).
\end{eqnarray*}

For best approximation by one linear form we defined $M_l = \max_{1\leq i \leq n} |z_{i,l}|$, and after hyperbolic rotation we have
\[ \max_{1\leq i \leq n} |z'_{i,l}| \leq M_l T^{-1}. \]

Since we assume that the best approximation vectors $(\bz_l)_{l\in\bbN}$ span the full space $\bbR^{n+1}$, we can apply Lemma \ref{lemnnb}  to $(\bz_l)_{l\in\bbN}$ and obtain a set of indices $(r_k)_{0 \leq 2^{n-2}-1}$. Denote
\[ \mathcal{S}'_{3,l} = \{ \bz'_{r_l-1}, \bz'_{r_l}, \bz'_{r_l+1} \} ={\cal G}_T { L} \mathcal{S}_{3,l} , \; \; \; 0 \leq l \leq 2^{n-2}-1 \]
and for $4 \leq k \leq n+1$ and $ 0\leq l \leq 2^{n-k+1} -1$ , denote by $\mathcal{S}'_{k,l}$ the set of best approximation vectors
\[ \mathcal{S}'_{k,l} = \cup_{\nu=0}^{2^{k-3}-1}  \mathcal{S}'_{3,2^{k-3}l + \nu} ={\cal G}_T { L}\mathcal{S}_{k,l} . \]

Since ${\cal G}_T$ and ${\cal L}$ have determinant $1$, these sets satisfies the properties of linear independence and inclusion listed in Lemma \ref{lemnnb}.\\

Further in the proof of Theorem \ref{MainThm}, we need an estimate on the fundamental volumes of the lattices $\Lambda_k'=\langle \bz'_k , \bz'_{k+1}   \rangle_\bbZ$ and $\Gamma_k'=\langle \bz'_{k-1} , \bz'_k , \bz'_{k+1}   \rangle_\bbZ$ spanned by consecutive independent vectors $\bz'_l$. For large $T$, we can follow a similar proof as in Lemma \ref{estim}.

\begin{lem}\label{estimrot}
Fix an index $k$. Let $T$ be large enough so that 
\begin{equation}\label{defT}
T> M_{k+1} \; \textrm{ and } \; T > L_{k-1}^{-1/n}.
\end{equation}
Given two consecutive and linearly independent best approximation vectors $\bz_k , \bz_{k+1}$, the fundamental volume $\det\Lambda_k'$ satisfies

\begin{equation}\label{uma01}
\det\Lambda_k' \asymp  L_k T^n M_{k+1} T^{-1} = L_k  M_{k+1} T^{n-1}.
\end{equation}

Given three consecutive and linearly independent best approximation vectors $\bz_{k-1}, \bz_k , \bz_{k+1}$, the fundamental volume $\det\Gamma_k'$ satisfies

\begin{equation}\label{uma02}
\det{\Gamma_k'} \ll L_{k-1} T^n M_kT^{-1} M_{k+1}T^{-1} = L_{k-1}M_kM_{k+1} T^{n-2}.
\end{equation}
\end{lem}

\begin{proof}
For $T $ satisfying \eqref{defT}, we see that $\bz_k'={}^t(T^{-1}q_{1,k}, \ldots , T^{-1}q_{n,k}, T^nL_k)$ satisfies

\begin{equation}\label{sizeT} |T^n L_k| > 1 \; \textrm{  and  }  \;  |T^{-1}q_{i,k}| < 1 \; \textrm{ for } 1 \leq i \leq n.\end{equation}

Consider the $2 \times (n+1)$ matrix 
\begin{equation}\label{uma1}
{}^t\left(\begin{array}{cccc}z'_{1,k}& \ldots & z'_{n,k}& z'_{n+1,k} \\
z'_{1,k+1}& \ldots & z'_{n,k+1}& z'_{n+1,k+1}\end{array}\right) ={}^t \left(\begin{array}{cccc}T^{-1}q_{1,k}& \ldots &T^{-1}q_{n,k}& T^nL_k \\
T^{-1}q_{1,k+1}& \ldots & T^{-1}q_{n,k+1}& T^nL_{k+1}\end{array}\right)
\end{equation}
and the $3\times (n+1)$ matrix 
\begin{equation}\label{uma2}
{}^t\left(\begin{array}{cccc}z'_{1,k-1}& \ldots & z'_{n,k-1}& z'_{n+1,k-1}\\
z'_{1,k}& \ldots & z'_{n,k}& z'_{n+1,k} \\
z'_{1,k+1}& \ldots & z'_{n,k+1}& z'_{n+1,k+1}\end{array}\right) ={}^t \left(\begin{array}{cccc}T^{-1}q_{1,k-1}& \ldots &T^{-1}q_{n,k-1}& T^nL_{k-1} \\
T^{-1}q_{1,k}& \ldots &T^{-1}q_{n,k}& T^nL_k \\
T^{-1}q_{1,k+1}& \ldots & T^{-1}q_{n,k+1}& T^nL_{k+1}\end{array}\right).
\end{equation}

The rest of the proof is completely analogous to the proof of Lemma \ref{estim}. To obtain the upper bounds in \eqref{uma01} and \eqref{uma02} we need to get upper bounds for $2\times2$ minors of the matrix (\ref{uma1}) and for $3\times 3$ minors of the matrix (\ref{uma2}) by taking into account inequalities \eqref{sizeT}. This bounds will be of the form
\[
\text{$2\times2$ minors of (\ref{uma1})}\,\, \ll L_k T^n M_{k+1} T^{-1} = L_k  M_{k+1} T^{n-1}
\]
and
\[
\text{$3\times3$ minors of (\ref{uma2})}\,\, \ll L_{k-1} T^n M_kT^{-1} M_{k+1}T^{-1} = L_{k-1}M_kM_{k+1} T^{n-2}.
\]
Then application of Lemma \ref{det2} gives upper bounds in (\ref{uma01}) and (\ref{uma02}).


The lower bound for $\det\Lambda_k'$ from  (\ref{uma01}) follows from Minkowski's first convex body theorem as well,
analogously to the argument of the final part of the proof of Lemma 2. One should consider the symmetric  convex body
\[
\Pi  =\left\{
\bz :\,\,
\max_{1\le j \le n} |z_j|< M_{k+1},\,  |z_{n+1}|< L_k\right\}
\]
and its image
\[\mathcal{G}_T \Pi =
\left\{
\bz :\,\,
\max_{1\le j \le n} |z_j|< \frac{M_{k+1}}{T},\,  |z_{n+1}|< T^nL_k\right\}
.
\]
 It is clear that
 \[
 \mathcal{G}_T\Pi \cap   \mathcal{L}' =  \Pi\cap \mathcal{L} =\{ {\bf 0}\}.
 \]
 Consider the section
 \[
 P = \mathcal{G}_T \Pi  \cap \langle \bz_k',\bz_{k+1}'\rangle_\mathbb{R},
 \]
 by means of Minkowski's theorem, we obtain the upper estimate for its area \[
 {\rm area}\, P \le4 \Lambda_k'.
 \]
 The lower bound
 \[
 {\rm area}\, P\gg  L_k  M_{k+1} T^{n-1}
 \]
 comes from (\ref{defT}).

\end{proof}

Here, we need a large parameter $T$ to obtain a good upper bound for the minors. If $T=1$, such upper bounds are false.
\begin{rmq}
In the case of a lattice generated by both $\Lambda: = \langle \bz'_\nu, \bz'_{\nu+1} \rangle_\bbZ = \langle \bz'_{k-1}, \bz'_{k} \rangle_\bbZ$
we have
\begin{equation}\label{split2rot}
\det\Lambda \asymp  L_\nu  M_{\nu+1} T^{n-1} \asymp L_{k-1}  M_{k} T^{n-1}.
\end{equation}
\end{rmq}

\subsection{Proof of Theorem \ref{MainThm} for approximation by one linear form}

The proof in the case of approximation by one linear form follow the same steps as in the case of simultaneous approximation. Hence, we give a sketch of the proof in general, but
to make the ideas of the proof clearer, in Section \ref{4nbdual}  we give a very detailed proof in the simplest case of approximation to $4$ numbers. 
Idea of the argument comes from \cite{Mo*}. Note that by reversing time, we get two inequalities in term of coefficients, and two in term of linear forms.

\subsubsection{Proof in any dimension}

Consider $\btheta\in\bbR^n$ with $\bbQ$-linearly independent coordinates with $1$, and take $\alpha^*<\hat{\omega}(\btheta)$. Let $g^*= G^* (n,\alpha^*)$ be the unique positive root of the polynomial $R^*_{n,\alpha^*}$ defined in \eqref{polybis}, recall \eqref{positivitydual}. We define for $4\leq k \leq n$ the parameters
\begin{equation}\label{defyzn} z_{n-k} = \cfrac{R^*_{k-1,\alpha^*}(g^*)}{R^*_{k-2,\alpha^*}(g^*)} \;\; \textrm{ and } \;\; y_{n-k} = \cfrac{R^*_{k-1,\alpha^*}(g^*)}{g^*R^*_{k-2,\alpha^*}(g^*)}. \end{equation}
which satisfy the assumptions \eqref{inityz} and \eqref{recyz} of Lemma \ref{optiformula} because of the induction formula \eqref{aaa} and $R^*_{n,\alpha^*}(g^*) = 0$.

Considering a sequence $(\bz_l)_{l\in\bbN}$ of best approximations to a point $\btheta\in \bbR^n$, we obtain via Lemma \ref{lemnnb} a set of indices satisfying good properties. Suppose that $k_0$ is large enough so that for $\alpha^* < \hat{\omega}(\btheta)$.
 \begin{equation}\label{zz3'}
L_j \leq M_{j+1}^{-\alpha^*}  ,  \; \; \textrm{ for } \; \; j \geq k_0.
\end{equation}

For any fixed  $T\gg1$, the hyperbolic rotation $(\bz_l')_{l\in\bbN}=\mathcal{G}_T \mathcal{L} \cdot  (\bz_l)_{l\in\bbN}$ preserves the property of linear independence, and hence the structure of the pattern of best approximation vectors constructed in Lemma \ref{lemnnb}. We consider the rotated sets $ \mathcal{S}'_{k,l} ={\cal G}_T { L}\cdot \mathcal{S}_{k,l}$, $ \mathcal{Q}'_{k,l} ={\cal G}_T { L}\cdot \mathcal{Q}_{k,l}$ from the sets $\mathcal{S}_{k,l}$ and $\mathcal{Q}_{k,l}$ defined in Lemma \ref{lemnnb}. We denote respectively by $\mathfrak{S}'_{k,l}$ and $\mathfrak{Q}'_{k,l}$ the lattices of their $\mathcal{G}_T\mathcal{L}$-points. Section \ref{lfl8} explains that we can modify the proof of Lemma \ref{optiformula} so that
\begin{equation}\label{MainEq2}
\prod_{l=0}^{2^{n-4} -1} \left(\cfrac{\det\left(  \mathfrak{S}'_{3,4l}\right) \det\left(  \mathfrak{Q}'_{3,l} \right)^{1-y_{n-4}} }{ \det\left(  \mathfrak{Q}'_{2,2l} \right)}\right)^{w_{n-4,l}}\left(\cfrac{ \det\left(  \mathfrak{Q}'_{3,l} \right) ^{1-z_{n-4}}\det\left(  \mathfrak{S}'_{3,4l+3} \right) }{ \det\left(  \mathfrak{Q}'_{2,2l+1} \right)}\right)^{w_{n-4,l}'}  \gg 1
\end{equation}
where the parameters $y_{n-4}, z_{n-4}$ are defined in \eqref{defyzn} and $w_{n-4,l}, w_{n-4,l}'$ are defined by \eqref{recw} and satisfy \eqref{sumexpl1} and \eqref{sumexpl2}. \\

As for the proof of the analogue of Lemma \ref{Keylemma}, we express the denominators in two different ways. Indeed, 
\begin{eqnarray}
  \mathfrak{Q}'_{2,2l} &=& \langle \bz'_{r_{4l}}, \bz'_{r_{4l}+1} \rangle_\bbZ = \langle \bz'_{r_{4l+1}-1}, \bz'_{r_{4l+1}} \rangle_\bbZ\\
 \mathfrak{Q}'_{2,2l+1}  &=& \langle \bz'_{r_{4l+2}}, \bz'_{r_{4l+2}+1} \rangle_\bbZ = \langle \bz'_{r_{4l+3}-1}, \bz'_{r_{4l+3}} \rangle_\bbZ
\end{eqnarray}
and we write both $\det\left(  \mathfrak{Q}'_{2,2l} \right)$ and $\det\left(  \mathfrak{Q}'_{2,2l+1} \right)$ with their two expressions coming from \eqref{split2rot}. Given $s,t \in [0,1]$, analogously to Lemma \ref{Keylemma}, we define 
\begin{equation}\label{defg*}
g^*(s,t) =  \frac{(1-\alpha^*)s}{(1-\alpha^*)w^* (s,t)-s} = \frac{(1-\alpha^*)(1-w^*(s,t)-t)}{t}. \end{equation}
where the second equality comes from $w^* (s,t)\in [0,1]$ being  the root of the equation
  \begin{equation}\label{eqo*}
 w^{*2} -\left( 1-t- \cfrac{s}{\alpha^*-1}  \right) w^* - \cfrac{s}{\alpha^*-1} = 0.
 \end{equation}
 We obtain the analogue of \eqref{gay} for $g^*(s,t)$
 \begin{equation}\label{gay*}
   1/g^*(s,t)^2 -\left( \frac{1}{\alpha^*-1} +\frac{1-t}{s}\right) \cdot 1/g^*(s,t) - \frac{t}{s(\alpha^*-1)} = 0.
   \end{equation}
 from which we deduce the analogue of \eqref{guy2} and \eqref{guy3}, that is 
    \begin{eqnarray}\label{guy1*}
     s &=& \frac{R^*_{3,\alpha^*}(g^*)}{R^*_{2,\alpha^*}(g^*)}, \; \;      \textrm{ for } \; \;   g^* = g^*(1-s,1),\\  \label{guy2*}
      t &=& \frac{R^*_{3,\alpha^*}(g^*)}{g^*R^*_{2,\alpha^*}(g^*)}, \; \; \textrm{ for } \; \;     g^* = g^*(1,1-t). 
   \end{eqnarray}
 Consider 
 \[ w_1 = w^*(1, 1-y_{n-4} ) \;\; \textrm{ and } \;\; w_2 = w^*(1-z_{n-4},1)\]
 and the associated values $g_1=g^*(1, 1-y_{n-4} )$ and $g_2=g^*(1-z_{n-4},1)$.\\
From \eqref{guy1*} and  \eqref{guy2*}, following similar argument as in Section \ref{PfMthm}, we get the analogue of \eqref{equalg}:
  \begin{eqnarray}\label{defg*} 
 g^*&=& g_1 = \cfrac{\alpha^*-1}{(\alpha^*-1)w_1 +1} = \cfrac{(\alpha^*-1)(w_1 - y_{n-4})}{1-y_{n-4}}, \\
g^* &=&  g_2 =  \cfrac{(\alpha^*-1)(1-z_{n-4})}{(\alpha^*-1)w_2 +1-z_{n-4}} = (\alpha^*-1)w_2 .
  \end{eqnarray}
 
Applying the estimates of Lemma \ref{estimrot}, and weighting the two ways to write the denominators coming from \eqref{split2rot} with parameters $w_1$ and $w_2$, we get
\begin{eqnarray*}\label{MainEq3}
\prod_{l=0}^{2^{n-4} -1} \left(\cfrac{\left( L_{r_{4l}-1} M_{r_{4l}} M_{r_{4l}+1}T^{n-2}\right) \left(  L_{r_{4l+1}-1} M_{r_{4l+1}} M_{r_{4l+1}+1}T^{n-2} \right)^{1-y_{n-4}} }{ \left(  L_{r_{4l}} M_{r_{4l}+1} T^{n-1}\right)^{w_1} \left(  L_{r_{4l+1}-1} M_{r_{4l+1}}T^{n-1}\right)^{1-w_1}}\right)^{w_{n-4,l}}   &\cdot \\
\prod_{l=0}^{2^{n-4} -1} \left(\cfrac{\left( L_{r_{4l+2}-1} M_{r_{4l+2}} M_{r_{4l+2}+1}T^{n-2}\right)^{1-z_{n-4}} \left(  L_{r_{4l+3}-1} M_{r_{4l+3}} M_{r_{4l+3}+1}T^{n-2} \right) }{ \left(  L_{r_{4l+2}} M_{r_{4l+2}+1} T^{n-1}\right)^{w_2} \left(  L_{r_{4l+3}-1} M_{r_{4l+3}}T^{n-1}\right)^{1-w_2}}\right)^{w'_{n-4,l}}  &\gg& 1\\
\end{eqnarray*}

Furthermore, by \eqref{sumexpl1} and \eqref{sumexpl2}, $T$ has the same power $(n-1)(n-2)$ at numerator and denominator and can be cancelled out.
\begin{eqnarray*}\label{MainEqnoT}
\prod_{l=0}^{2^{n-4} -1} \left(\cfrac{\left( L_{r_{4l}-1} M_{r_{4l}} M_{r_{4l}+1}\right) \left(  L_{r_{4l+1}-1} M_{r_{4l+1}} M_{r_{4l+1}+1} \right)^{1-y_{n-4}} }{ \left(  L_{r_{4l}} M_{r_{4l}+1} \right)^{w_1} \left(  L_{r_{4l+1}-1} M_{r_{4l+1}}\right)^{1-w_1}}\right)^{w_{n-4,l}}   &\cdot \\
\prod_{l=0}^{2^{n-4} -1} \left(\cfrac{\left( L_{r_{4l+2}-1} M_{r_{4l+2}} M_{r_{4l+2}+1}\right)^{1-z_{n-4}} \left(  L_{r_{4l+3}-1} M_{r_{4l+3}} M_{r_{4l+3}+1}\right) }{ \left(  L_{r_{4l+2}} M_{r_{4l+2}+1} \right)^{w_2} \left(  L_{r_{4l+3}-1} M_{r_{4l+3}}\right)^{1-w_2}}\right)^{w'_{n-4,l}}  &\gg& 1\\
\end{eqnarray*}

Hence, at least \emph{one} of the following four inequalities holds:
\begin{eqnarray}\label{dd10} 
L_{r_{4l}-1} M_{r_{4l}} M_{r_{4l}+1} &\gg& \left( L_{r_{4l}} M_{r_{4l}+1} \right)^{w_1},  \\\label{dd20}
\left(  L_{r_{4l+1}-1} M_{r_{4l+1}} M_{r_{4l+1}+1} \right)^{1-y_{n-4}} &\gg& \left(  L_{r_{4l+1}-1} M_{r_{4l+1}}\right)^{1-w_1},\\\label{dd30}
\left( L_{r_{4l+2}-1} M_{r_{4l+2}} M_{r_{4l+2}+1}\right)^{1-z_{n-4}}  &\gg& \left(  L_{r_{4l+2}} M_{r_{4l+2}+1} \right)^{w_2}, \\\label{dd40}
  L_{r_{4l+3}-1} M_{r_{4l+3}} M_{r_{4l+3}+1}  &\gg& \left(  L_{r_{4l+3}-1} M_{r_{4l+3}}\right)^{1-w_2}.
\end{eqnarray}

Using \eqref{zz3'} and \eqref{defg*}, we deduce that  

\begin{enumerate}
\item{inequality  (\ref{dd10})  leads to  $L_{r_{4l}} \ll M_{r_{4l}}^{-\alpha^*g^*}$;}
\item{ inequality (\ref{dd20}) leads to  $M_{r_{4l+1}+1} \gg  M_{r_{4l+1}}^{g^*}$;}
\item{inequality  (\ref{dd30}) leads to $L_{r_{4l+2}}\ll M_{r_{4l+2}}^{-\alpha^*g^*}$;}
\item{inequality  (\ref{dd40}) leads to  $M_{r_{4l+3}+1} \gg M_{r_{4l+3}}^{g^*}$.}
\end{enumerate} 
 
We explain how to get the first two inequalities of this group from the first  two  inequality of the previous group. The others are obtained in a similar way, however one should note that
the inequality
\eqref{positivitydual}
 is crucial for checking the positivity of exponents in case 3). 

1) Indeed, suppose (\ref{dd10}). Then 
  as $L_{r_{4l}} < M_{r_{4l}+1}^{-\alpha^*}$ or $M_{r_{l4}+1}< L_{r_{4l}}^{-1/\alpha^*} $, we deduce the upper bound for the linear form  $ L_{r_{4l}}$ by means of the estimate
\[
L_{r_{4l}}^{w_1} \ll  L_{r_{4l}-1} M_{r_{4l}}  M_{r_{4l}+1}^{1-w_1} \ll  M_{r_{4l}}^{1-\alpha^*} L_{r_{4l}}^{-\frac{1-w_1}{\alpha^*}}
  \;\;\textrm {, or}\;\;
 L_{r_{4l}}  \ll M_{r_{4l}}^{ \frac{1-\alpha^*}{w_1+(1-w_1)/\alpha^*}}  .
\]
 We use the first inequality form (\ref{defg*}) to conclude that 
 $L_{r_{4l}} \ll M_{r_{4l}}^{-\alpha^*g^*}$.
 
 2) Suppose (\ref{dd20}). Then 
 we use $L_{r_{4l+1}-1}M_{r_{4l+1}} < M_{r_{4l+1}}^{1-\alpha^*}$.
 Now
\[ 1 \ll  M_{r_{4l+1}+1}^{1-y_{n-4}} (L_{r_{4l+1}-1}M_{r_{4l+1}})^{{w_1-y_{n-4}}} \ll M_{r_{4l+1}+1}^{1-y_{n-4}}  M_{r_{4l+1}}^{(1-\alpha^*)({w_1-y_{n-4}})}\]
(here 
we use the inequality $ w_1-y_{n-4}>0$ which
follows from  $g^*>0$ and the second inequality form (\ref{defg*})).
We use the second inequality form (\ref{defg*}) to conclude that
\[
M_{r_{4l+1} +1} \gg M_{r_{4l+1}} ^{(\alpha^*-1)\frac{w_1-y_{n-4}}{1-y_{n-4}}} = M_{r_{4l+1}}^{g^*}.\]
 
 We have checked \eqref{equivthmdual} and the result follows.\qed

 
%
 %

\subsubsection{Example of approximation to 4 numbers}\label{4nbdual}
Consider a sequence of best approximation vectors to $\btheta\in\bbR^4$ by one linear form. We may assume that it spans $\bbR^5$. Take $\alpha^*<\hat{\omega}(\btheta)$.\\

We consider the unique positive real number $g^*$ such that $\alpha^*-1-g^*-(g^*)^2-(g^*)^3=0$. Set 
\[ x := \frac{\alpha^*-1-g^*-(g^*)^2}{\alpha^*-1-g^*} = \cfrac{\alpha^*-1-\alpha^*g^*}{g^*(g^*-\alpha^*+1)} = \cfrac{R^*_{3,\alpha^*}(g^*)}{R^*_{2,\alpha^*}(g^*)} = 1-\cfrac{R^*_{3,\alpha^*}(g^*)}{g^*R^*_{2,\alpha^*}(g^*)}.\]
Set the parameters (using \eqref{eqo*})   
\[  w_1 = w^*(1,x)=  \frac{\alpha^*-1-g^*}{g^*(\alpha^*-1)} \;\; \textrm{ and } \;\; w_2 = w^*(1-x,1) = \frac{g^*}{\alpha^*-1}.\]
One can check that 
\begin{equation}\label{Gstar}
g^* = \cfrac{\alpha^*-1}{(\alpha^*-1)w_1+1} = (\alpha^*-1)(1+(w_1-1)/x) = \cfrac{(\alpha^*-1)(1-x)}{(\alpha^*-1)w_2 +1-x} = (\alpha^*-1)w_2.
\end{equation}
As $0<g^*=(\alpha^*-1)\cfrac{x+w_1-1}{x}$, we deduce that 
\begin{equation}\label{W1}
w_1 +x-1 > 0.
\end{equation}
As $\cfrac{\alpha^*-1}{\alpha^*} < 1 \leq g^* \leq \alpha^*-1$ we have 
\begin{equation}\label{W2}
1-x - w_2 = 1-\cfrac{ R_{3,\alpha^*}^*(g^*)}{R^{*}_{2,\alpha^*}(g^*)} - w_2 = \cfrac{  \alpha^* g^* \left(g^*-  \cfrac{\alpha^*-1}{\alpha^*} \right)   }{ (\alpha^*-1)(\alpha^*-1-g^*)  } >0
\end{equation}

Now we are able to start the proof. For an index $k_0\gg1$ we apply Lemma \ref{lem4nb}. It provides a pattern of best approximation vectors
  $$
  \pmb{z}_{r_0-1},
  \pmb{z}_{r_0},
  \pmb{z}_{r_0+1};
  \,\,\,\,\,
   \pmb{z}_{r_1-1},
  \pmb{z}_{r_1},
  \pmb{z}_{r_1+1};
  \,\,\,\,\,
   \pmb{z}_{r_2-1},
  \pmb{z}_{r_2},
  \pmb{z}_{r_2+1};
   \,\,\,\,\,
   \pmb{z}_{r_3-1},
  \pmb{z}_{r_3},
  \pmb{z}_{r_3+1};
  $$
 of linearly independent triples satisfying properties of Lemma \ref{lem4nb}. Consider $T$ such that $T> M_{r_3+1} \; \textrm{ and } \; T > L_{r_3-1}^{-1/n}$. For $j\ge r_0-1$, we apply the hyperbolic rotation to the integer vectors $\pmb{z}_{j}$ to get 
 \[  \pmb{z}_{j}' =   {\cal G}_T {L}\cdot \pmb{z}_{j}.\]
 
 For $0\leq i \leq 3$ we consider the subspace 
 \[\mathbf{S}_{3,i} = \langle   \pmb{z}'_{r_i-1},  \pmb{z}'_{r_i},  \pmb{z}'_{r_i+1} \rangle_\bbR\]
 and its lattice of $\mathcal{G}_T\mathcal{L}$ points 
 \[\mathfrak{S}_{3,i}=\mathbf{S}_{3,i} \cap \mathcal{G}_T\mathcal{L}.\] 
 We recall that \[ \mathbf{S}_{3,1} = \mathbf{S}_{3,2} = \mathbf{Q}. \]
 Consider the $2$-dimensional lattices
\[\Lambda_0 : = \langle \bz'_{r_0}, \bz'_{r_0+1} \rangle_{\bbZ} = \langle \bz'_{r_1-1}, \bz'_{r_1} \rangle_{\bbZ} =\mathbf{S}_{3,0} \cap \mathbf{S}_{3,1} \cap \mathcal{G}_T\mathcal{L}\]
and
\[\Lambda_1 := \langle \bz'_{r_2}, \bz'_{r_2+1} \rangle_{\bbZ} = \langle \bz'_{r_3-1}, \bz'_{r_3}\rangle_{\bbZ}= \mathbf{S}_{3,2}\cap \mathbf{S}_{3,3} \cap \mathcal{G}_T\mathcal{L}.\]

We apply Schmidt's inequality (Propositon \ref{SchmidtHeightGen}) with underlying lattice $\mathcal{G}_T\mathcal{L}$ to obtain the analogue of \eqref{sh4}
  $$
  \frac{{\rm det}\, \mathfrak{ S}_{3,0}
  ( {\rm det}\,\mathfrak{S}_{3,1} )^x}{{\rm det}\,\Lambda_0}
  \cdot
    \frac{({\rm det}\, \mathfrak{S}_{3,2})^{1-x} 
   {\rm det}\,\mathfrak{ S}_{3,3} }{{\rm det}\,\Lambda_1} \gg 1.
  $$
  By Lemma \ref{estimrot}, we get
  $$
  \frac{   L_{r_0-1}M_{r_0}M_{r_0+1} T^{2}
  ( L_{r_1-1}M_{r_1}M_{r_1+1} T^{2} )^x}{ L_{r_0} M_{r_0+1} T^{3}}
  \cdot
    \frac{(L_{r_2-1}M_{r_2}M_{r_2+1} T^{2})^{1-x} 
   L_{r_3-1}M_{r_3}M_{r_3+1} T^{2}}{L_{r_3-1} M_{r_3} T^{3}} \gg 1.
  $$
  Here, $T$ disappears as it has power $6$ at numerator and denominator :
  \[ 3+3 =6 = 2 +2x +2(1-x) +2.\]
  We deduce
   $$
  \frac{   L_{r_0-1}M_{r_0}M_{r_0+1} 
  ( L_{r_1-1}M_{r_1}M_{r_1+1} )^x}{ L_{r_0} M_{r_0+1} }
  \cdot
    \frac{(L_{r_2-1}M_{r_2}M_{r_2+1} )^{1-x} 
   L_{r_3-1}M_{r_3}M_{r_3+1} }{L_{r_3-1} M_{r_3} } \gg 1.
  $$
$\Lambda_0 = \langle \bz'_{r_0}, \bz'_{r_0+1}\rangle_\mathbb{Z} = \langle \bz'_{r_1-1}, \bz'_{r_1}\rangle_\mathbb{Z}$ therefore $L_{r_0} M_{r_0+1} \asymp L_{r_1-1} M_{r_1}$ according to \eqref{split2rot} and by analogous arguments applied to $\Lambda_1$ we get the second equation $ L_{r_2} M_{r_2+1} \asymp L_{r_3-1} M_{r_3}$.
Hence we can replace \[L_{r_0} M_{r_0+1} \; \textrm{ by } \; (L_{r_0} M_{r_0+1})^{w_1}(L_{r_1-1} M_{r_1})^{1-w_1}\]
  and 
  \[L_{r_3-1} M_{r_3} \; \textrm{ by } \; (L_{r_2} M_{r_2+1})^{w_2}(L_{r_3-1} M_{r_3})^{1-w_2}.\]   
  We deduce that at least \emph{one} of the four following inequalities holds
\begin{eqnarray}\label{dd1}
  L_{r_0-1}M_{r_0}M_{r_0+1} &\gg& (L_{r_0} M_{r_0+1})^{w_1},  \\\label{dd2}
  ( L_{r_1-1}M_{r_1}M_{r_1+1} )^x &\gg&  (L_{r_1-1} M_{r_1})^{1-w_1}, \\ \label{dd3}
  (L_{r_2-1}M_{r_2}M_{r_2+1} )^{1-x} &\gg& (L_{r_2} M_{r_2+1})^{w_2}, \\ \label{dd4}
   L_{r_3-1}M_{r_3}M_{r_3+1} &\gg & (L_{r_3-1} M_{r_3})^{1-w_2}.
  \end{eqnarray}
  
  {\color{blue}
    1) } From \eqref{dd1} and \eqref{Gstar}, as $L_{r_0} < M_{r_0+1}^{-\alpha^*}$ or $M_{r_0+1}< L_{r_0}^{-1/\alpha^*} $, we deduce the upper bound for the linear form 
\begin{eqnarray*}
 L_{r_0}^{w_1} \ll   L_{r_0-1} M_{r_0}  M_{r_0+1}^{1-w_1} \ll  M_{r_0}^{1-\alpha^*} L_{r_0}^{-\frac{1-w_1}{\alpha^*}} \;\;\textrm {, or}\;\;
 L_{r_0}  \ll M_{r_0}^{ \frac{1-\alpha^*}{w_1+(1-w_1)/\alpha^*}} = M_{r_0}^{-\alpha^*g^*}.
 \end{eqnarray*}

  {\color{blue}
    2)  } From \eqref{dd2} and \eqref{Gstar}, as $L_{r_1-1}M_{r_1} < M_{r_1}^{1-\alpha^*}$, we deduce the lower bound for the coefficient
\[ 1 \ll  M_{r_1+1}^x (L_{r_1-1}M_{r_1})^{x+w_1-1} \ll M_{r_1+1}^x  M_{r_1}^{(1-\alpha^*)(x+w_1-1)}\textrm {, or}\;\;
M_{r_1 +1} \gg M_{r_1} ^{(\alpha^*-1)\frac{x+w_1-1}{x}} = M_{r_1}^{g^*}.\]
The second inequality is satisfied because of \eqref{W1}.

3)   From \eqref{dd3} and \eqref{Gstar}, as $L_{r_2-1} < M_{r_2}^{-\alpha^*}$ or $M_{r_2+1}< L_{r_2}^{-1/\alpha^*} $, we deduce the upper bound for the linear form 
\[ L_{r_2}^{w_2} \ll (L_{r_2-1}M_{r_2})^{1-x} M_{r_2+1}^{1-x-w_2} \ll M_{r_2}^{(1-x)(1-\alpha^*)} L_{r_2}^{\frac{x+w_2-1}{\alpha^*}}
\textrm {, or}\;\;L_{r_2} \ll M_{r_2}^{\frac{(1-\alpha^*)(1-x)}{w_2+(1-x-w_2)/\alpha^*}} = M_{r_2}^{-\alpha^*g^*}.\]
Here we use \eqref{W2}.

4) From \eqref{dd4} and \eqref{Gstar}, as $L_{r_3-1}M_{r_3} < M_{r_3}^{\alpha^*-1}$, we deduce the lower bound for the coefficient
\[  1 \ll (L_{r_3-1}M_{r_3})^{w_2}M_{r_3+1} \ll M_{r_3}^{(\alpha^*-1)w_2} M_{r_3+1}
\textrm {, or}\;\;M_{r_3 +1} \gg M_{r_3} ^{\textcolor{green}{(\alpha^*-1)}w_2} = M_{r_1}^{g^*}.\]


Hence, we proved that one of the following four inequalities holds:
\begin{equation*}
  M_{r_1+1} \gg  M_{r_1}^{g^*}, \;\; M_{r_3+1} \gg M_{r_3}^{g^*},\;\;  L_{r_0} \ll M_{r_0}^{-\alpha^*g^*},\;\;  L_{r_2}\ll M_{r_2}^{-\alpha^*g^*}.
  \end{equation*}
  So we have checked \eqref{equivthmdual} and the result follows.\qed
  
 
\section{Construction of points with given ratio}\label{PGN}

In this last section, we prove the second part of Theorem \ref{MainThm}. To construct points with given ratio, we place ourselves in the context of parametric geometry of numbers introduced by Schmidt and Summerer in \cite{SSfirst, SS}. For the convenience of the reader and the sake of self-containment, we briefly present the parametric geometry of numbers  in section \ref{PGN-intro}. An important theorem by Roy \cite{Roy} enables to construct points with computable exponents of Diophantine approximation out of \emph{Roy-systems}, a combinatorial family of piecewise linear applications. For our purpose, we construct explicitly in Section \ref{PGN-constr} a family of Roy-systems with three parameters. The construction shows how the values $G(n,\alpha)$ and $G^*(n,\alpha^*)$ appear naturally in the context of parametric geometry of numbers, and why they are reached at \emph{regular systems}.\\

\subsection{Parametric geometry of numbers}\label{PGN-intro}

The Parametric Geometry of Numbers answers a question of W. M. Schmidt \cite{SchLN}. Given a convex body and a lattice, we deform either of them with a one parameter diagonal map. We study the behavior of the successive minima in terms of this parameter. It was developed by W. M. Schmidt and L. Summerer \cite{SS,SSfirst}, and further by D.Roy \cite{Roy}\footnote{In \cite{DFSU17,VarPrinc}, Das, Fishman, Simmons and Urba\'nski introduce a variational principle in parametric geometry of numbers that extends Theorem \ref{DR}. They both extend to the case of approximation to a matrix $\btheta$, and provide a quantitative result. Applying the variational principle to our construction, we obtain a lower bound for the Hausdorff dimension of points with given pair of exponents $(\omega,\hat{\omega})$ or $(\lambda, \hat{\lambda})$ satisfying \eqref{e-Mthm}. However, for $c>1$ it is probably not optimal.}.
 
We use the notation introduced by D. Roy in \cite{Roy} which is essentially dual to the one of W. M. Schmidt and L. Summerer \cite{SS}. It follows the presentation in \cite[\S2]{Mar}. We refer the reader to these papers for further details.\\

 Here $ \boldsymbol{x}\cdot \boldsymbol{y}=x_1y_1+ \cdots + x_ny_n$  is the usual scalar product of vectors $\boldsymbol{x}$ and $\boldsymbol{y}$, and $\|\boldsymbol{x}\|_2=\sqrt{\boldsymbol{x}\cdot \boldsymbol{x}}$ is the usual Euclidean norm.\\

Let $\boldsymbol{u} = (u_0, \ldots , u_n)$ be a vector in $\mathbb{R}^{n+1}$, with Euclidean norm $\|\boldsymbol{u}\|_2=1$. For a real parameter $Q\geq1$ we consider the convex body
\[ \mathcal{C}_{\boldsymbol{u}}(Q) = \left\{ \boldsymbol{x} \in \mathbb{R}^{n+1} \mid \|\boldsymbol{x}\|_2 \leq 1 , \; |\boldsymbol{x} \cdot \boldsymbol{u} | \leq Q^{-1}     \right\}. \]
 For $1\leq d \leq n+1$ we denote by $\lambda_d\left( \mathcal{C}_{\boldsymbol{u}}(Q) \right)$ the $d$-th minimum of $\mathcal{C}_{\boldsymbol{u}}(Q)$ relatively to the lattice $\mathbb{Z}^{n+1}$. For $q\geq0$ and $1\leq d \leq n+1$ we set
\[ L_{\boldsymbol{u},d}(q) = \log   \lambda_d\left( \mathcal{C}_{\boldsymbol{u}}(e^q) \right). \]
Finally, we define the \emph{successive minima function} $\boldsymbol{L_u}$ associated with $\boldsymbol{u}$:

\[ \begin{array}{rccl}
 \boldsymbol{L_u} :& [0,\infty ) & \to& \mathbb{R}^{n+1}  \\
  & q & \mapsto& (L_{\boldsymbol{u},1}(q), \ldots , L_{\boldsymbol{u},n+1}(q))  .
  \end{array}\]
  The lattice $\mathbb{Z}^{n+1}$ is invariant under permutation of coordinates. Hence,  $\boldsymbol{L_u}$ remains the same if we permute the coordinates in $\boldsymbol{u}$. Since $\|\boldsymbol{u}\|_2=1$ we can thus assume that $u_0\neq0$. \\
  
The following proposition links the exponents of Diophantine approximation associated with $\boldsymbol{\theta}=\left(\cfrac{u_1}{u_0}, \ldots , \cfrac{u_n}{u_0}\right)$ to the behavior of the map $ \boldsymbol{L_u}$, assuming $u_0\neq0$. 

\begin{prop}\label{prop}
Let $\boldsymbol{u} = (u_0, \ldots , u_n) \in \mathbb{R}^{n+1}$, with Euclidean norm $\|\boldsymbol{u}\|_2=1$ and  $u_0\neq 0$. Set $\boldsymbol{\theta}=\left(\cfrac{u_1}{u_0}, \ldots , \cfrac{u_n}{u_0}\right)$. We have the following relations:

\begin{eqnarray*}
\liminf_{q\to+\infty}  \cfrac{ L_{\boldsymbol{u},1}(q) }{q} = \cfrac{1}{1+{\omega}(\boldsymbol{\theta})} &,&  \limsup_{q\to+\infty} \cfrac{ L_{\boldsymbol{u},1}(q)}{q} = \cfrac{1}{1+\hat{\omega}(\boldsymbol{\theta})},\\
\limsup_{q\to+\infty}  \cfrac{ L_{\boldsymbol{u},n+1}(q) }{q} = \cfrac{\lambda}{1+{\lambda}(\boldsymbol{\theta})} &,& \liminf_{q\to+\infty} \cfrac{ L_{\boldsymbol{u},n+1}(q)}{q} = \cfrac{\hat{\lambda}}{1+\hat{\lambda}(\boldsymbol{\theta})}.
\end{eqnarray*}
\end{prop}

Thus, if we know an explicit map $\boldsymbol{P}=(P_1, \ldots, P_{n+1}): [0,\infty) \to \mathbb{R}^{n+1}$, such that 
$\boldsymbol{L}_{\boldsymbol{u}}-\boldsymbol{P}$ is bounded, then we can compute the $4$ exponents $\hat{\omega}(\boldsymbol{\theta}), \hat{\lambda}(\boldsymbol{\theta}), {\omega}(\boldsymbol{\theta}), {\lambda}(\boldsymbol{\theta})$ for the above point $\boldsymbol{\theta}$ upon replacing $L_{\boldsymbol{u},i}$ by $P_i$ in the above formulas for $ i= 1$ or $n+1$. For this purpose, Roy introduced \cite{Roy} what we will call \emph{Roy-systems}.

\begin{defn}
Let $I$ be a subinterval of $[0,\infty)$ with non-empty interior. A generalized $(n+1)$-system on $I$ is a continuous piecewise linear map $\boldsymbol{P} = (P_1, \ldots , P_{n+1}): I \to \mathbb{R}^{n+1}$ with the following three properties.

\begin{description}

\item[(S1)]{For each $q\in I$, we have $0\leq P_1(q) \leq \cdots \leq P_{n+1}(q) $  and $P_1(q) + \cdots + P_{n+1}(q) =q$.}

\item[(S2)]{ If $H$ is a non-empty open subinterval of $I$ on which $\boldsymbol{P}$ is differentiable, then there are integers $\underline{r}, \bar{r}$ with $1\leq \underline{r} \leq \bar{r} \leq n+1$ such that $P_{\underline{r}}, P_{\underline{r}+1}, \ldots , P_{\bar{r}}$ coincide on the whole interval $H$ and have slope $1/(\bar{r}-\underline{r}+1)$ while any other component $P_k$ of $\boldsymbol{P}$ is constant on $H$ .}

\item[(S3)]{ If $q$ is an interior point of $I$ at which $\boldsymbol{P}$ is not differentiable, if $\underline{r}, \bar{r}, \underline{s},\bar{s}$ are the integers for which
\[ P_k'(q^-) = \cfrac{1}{\bar{r}-\underline{r}+1}  \quad  (\underline{r} \leq k \leq \bar{r}) \quad \textrm{  and  }  \quad P_k'(q^+) = \cfrac{1}{\bar{s}-\underline{s}+1}  \quad  (\underline{s} \leq k \leq \bar{s}) \; , \]
and if $\underline{r}<\bar{s}$, then we have $P_{\underline{r}}(q) = P_{\underline{r}+1}(q) = \cdots = P_{\bar{s}}(q)$.
 }\\
\end{description}
\end{defn}

Here $P_k'(q^-)$ (resp. $P_k'(q^+)$) denotes the left (resp. right) derivative of $P_k$ at $q$.

\begin{thm}[Roy, 2015]\label{DR}
For each non-zero point $\boldsymbol{u} \in \mathbb{R}^{n+1}$, there exists $q_0\geq 0$ and a generalized $(n+1)$-system $\boldsymbol{P}$ on $[q_0,\infty)$ such that $\boldsymbol{L_u} - \boldsymbol{P}$ is bounded on $[q_0,\infty)$. Conversely, for each generalized $(n+1)$-system $\boldsymbol{P}$ on an interval $[q_0,\infty)$ with $q_0\geq0$, there exists a non-zero point $\boldsymbol{u}\in\mathbb{R}^{n+1}$ such that $\boldsymbol{L_u} - \boldsymbol{P}$ is bounded on $[q_0,\infty)$.\\
\end{thm}

In view of the remark following Proposition \ref{prop}, this result reduces the construction of points with prescribed exponents of Diophantine approximation to a combinatorial study of Roy-systems.\\

\subsection{Construction of a family of Roy-systems with three parameters}\label{PGN-constr}

In this section, we construct explicitly a family of Roy-systems with parameters. According to Proposition \ref{prop} and Theorem \ref{DR}, these Roy-systems provide the existence of points with requested exponents, proving the second part of Theorem \ref{MainThm}.

\paragraph{Approximation by one linear form.}
Fix the dimension $n\geq2$, and consider the case of approximation by one linear form. Fix the three parameters $\hat{\omega}\geq n$, $\rho=G^*(n,\hat{\omega})$ and $c \geq 1$. Consider the Roy-system $\boldsymbol{P}$ on the interval $[1,c\rho]$ depending on these parameters whose combined graph is given below by Figure \ref{PGN1}, where
\[ P_1(1) = \cfrac{1}{1+ \hat{\omega}}, \; \; P_k(1) = \rho^{k-2}P_1(1)\;  \textrm{ for } 2\leq k \leq n+1 \textrm{ and } P_{k}(c\rho)=c\rho P_{k}(1) \textrm{ for } 1\leq k \leq n+1.\]

\begin{figure}[!h] 
 \begin{center}
 \begin{tikzpicture}[scale=0.6]

\draw (3.5,3.5) node {$\cdots$};
\draw (8.5,3.5) node {$\cdots$};
\draw (16.5,3.5) node {$\cdots$};
\draw (19.5,3.5) node {$\cdots$};

\draw (3.5,8.5) node {$\cdots$};
\draw (8.5,8.5) node {$\cdots$};
\draw (16.5,8.5) node {$\cdots$};
\draw (19.5,8.5) node {$\cdots$};

\draw (16.5,12.5) node {$\cdots$};
\draw (19.5,12.5) node {$\cdots$};
 
\draw[black, semithick] (3,3)--(0,0) node [above,black] {};
\draw[black, semithick] (4,4)--(8,8) node [above,black] {};
\draw[black, semithick] (9,9)--(14,14) node [above,black] {};
\draw[black, semithick] (14,10)--(15,11) node [left,black] {};
\draw[black, semithick] (17, 7)--(18,8) node [left,black] {};
\draw[black, semithick] (18, 5)--(19,6) node [left,black] {};
\draw[black, semithick] (20, 2)--(21,3) node [left,black] {};
\draw[black, semithick] (21, 1)--(23,3) node [left,black] {};

\draw[black, semithick] (3,1)--(0,1) node [left,black] {$\frac{1}{1+ \hat{\omega}}$};
\draw[black, semithick] (4,1)--(8,1) node [left,black] {};
\draw[black, semithick] (16,1)--(9,1) node [left,black] {};
\draw[black, semithick] (17,1)--(19,1) node [left,black] {};
\draw[black, semithick] (20,1)--(21,1) node [left,black] {};

\draw[black, semithick] (3,2)--(0,2) node [left,black] {$\frac{\rho}{1+ \hat{\omega}}$};
\draw[black, semithick] (4,2)--(8,2) node [left,black] {};
\draw[black, semithick] (16,2)--(9,2) node [left,black] {};
\draw[black, semithick] (17,2)--(19,2) node [left,black] {};

\draw[black, semithick] (3, 5)--(0,5) node [left,black] {$\frac{\rho^{k-1}}{1+ \hat{\omega}}$};
\draw[black, semithick] (4,5)--(8,5) node [left,black] {};
\draw[black, semithick] (16,5)--(9,5) node [left,black] {};
\draw[black, semithick] (17,5)--(18,5) node [left,black] {};

\draw[black, semithick] (3,7)--(0,7) node [left,black] {$\frac{\rho^k}{1+ \hat{\omega}}$};
\draw[black, semithick] (4,7)--(8,7) node [left,black] {};
\draw[black, semithick] (16,7)--(9,7) node [left,black] {};

\draw[black, semithick] (14,14)--(16,14) node [left,black] {};
\draw[black, semithick] (17,14)--(19,14) node [left,black] {};
\draw[black, semithick] (20,14)--(23,14) node [right,black] {$\frac{c\rho^n}{1+ \hat{\omega}}$};

\draw[black, semithick] (16,11)--(15,11) node [left,black] {};
\draw[black, semithick] (17,11)--(19,11) node [left,black] {};
\draw[black, semithick] (20,11)--(23,11) node [right,black] {$\frac{c\rho^{n-1}}{1+ \hat{\omega}}$};

\draw[black, semithick] (20,8)--(23,8) node [right,black] {$\frac{c\rho^{k+1}}{1+ \hat{\omega}}$};
\draw[black, semithick] (19,8)--(18,8) node [left,black] {};

\draw[black, semithick] (20,6)--(23,6) node [right,black] {$\frac{c\rho^{k}}{1+ \hat{\omega}}$};

\draw[black, semithick] (21,3)--(23,3) node [right,black] {$\frac{c\rho}{1+ \hat{\omega}}$};

\draw[black, semithick] (3,10)--(0,10) node [left,black] {$\frac{\rho^{n-1}}{1+ \hat{\omega}}$};
\draw[black, semithick] (4,10)--(8,10) node [left,black] {};
\draw[black, semithick] (14,10)--(9,10) node [left,black] {};

\draw[dashed, black] (1,10)--(1,0) node [below] { $1$};
\draw[dashed, black] (23,14)--(23,0) node [below] { $c\rho$};
\draw[dashed, black] (21,14)--(21,0) node [below] { $q_1$};
\draw[dashed, black] (14,14)--(14,0) node [below] { $q_0$};

 \end{tikzpicture}
 \end{center}
 \caption{Pattern of the combined graph of $\boldsymbol{P}$ on the fundamental interval $[1,c\rho]$ }\label{PGN1}
 \end{figure}
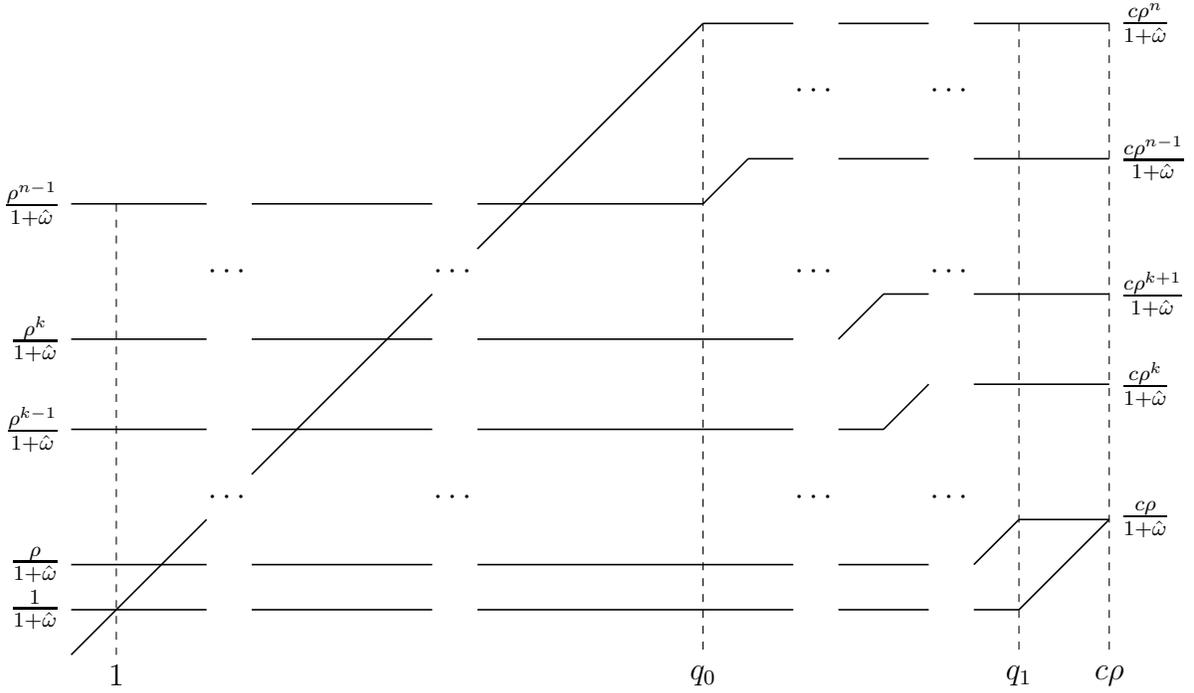

The fact that all coordinates sum up to $1$ for $q=1$ follows from $\rho$ being a root of the polynomial $R^*_{n,\homega}$ defined in \eqref{polybis}. On each interval between two consecutive division points, there is only one line segment with slope $1$. On $[1,q_0]$, there is one line segment of slope $1$ starting from the value $\frac{1}{1+ \hat{\omega}}$ and reaching the value $\frac{c \rho^{n}}{1+ \hat{\omega}}$. Then, each component $P_k$ increases from $\frac{\rho^{k-1}}{1+ \hat{\omega}}$ to $\frac{c\rho^{k-1}}{1+ \hat{\omega}}$ with slope $1$ where $k$ decreases from $k=n$ down to $k=2$.\\

We extend $\boldsymbol{P}$ to the interval $[1,\infty)$ by self-similarity. This means, $\boldsymbol{P}(q)=(c\rho)^{m}\boldsymbol{P}((c\rho)^{-m}q)$ for all integers $m$. In view of the value of $\boldsymbol{P}$ and its derivative at $1$ and $c\rho$, one sees that the extension provides a Roy-system on $[1,\infty)$.\\

Note that for $c=1$, the parameters $q_0$ and $q_1$ coincide and we constructed a \emph{regular system}.\\

Roy's Theorem \cite{Roy} provides the existence of a point $\boldsymbol{\theta}$ in $\mathbb{R}^{n}$ such that 
\begin{equation*}
\cfrac{1}{1+\hat{\omega}(\boldsymbol{\theta})} =  \limsup_{q\to+\infty} \cfrac{P_{1}(q)}{q} \;\; \textrm{ and } 
 \cfrac{1}{1+\omega(\boldsymbol{\theta})} = \liminf_{q\to+\infty} \cfrac{P_{1}(q)}{q}.
   \end{equation*}
Here, self-similarity ensures that the $\limsup$ (resp. $\liminf$) is in fact the maximum (resp. the minimum) on the interval $[1,c\rho[$. Thus,
\[ \cfrac{1}{1+\hat{\omega}(\boldsymbol{\theta})} =  \max_{[1,c\rho[} \cfrac{P_{1}(q)}{q} = \cfrac{P_1(1)}{1} = \cfrac{1}{1+ \hat{\omega}},\;\; \;\;
 \cfrac{1}{1+\omega(\boldsymbol{\theta})} = \min_{[1,c\rho[} \cfrac{P_{1}(q)}{q} = \cfrac{P_1(q_1)}{q_1} = \cfrac{1}{c\rho\hat{\omega} +1} 
\]
where \[q_1 = \cfrac{c(\rho^n + \cdots + \rho^2 + \rho) + 1}{1+ \hat{\omega}} = \cfrac{c(\rho\hat{\omega})+1}{1+ \hat{\omega}}.\] 
The simplification comes from $\rho$ being a root of the polynomial $R^*_{n,\homega}$ defined in \eqref{polybis}. Hence, $ \hat{\omega}(\btheta) = \hat{\omega}$ and $ \omega(\btheta) = c\rho \hat{\omega}$, and we constructed the required points since $c\geq1$ and $\rho=G^*(n,\hat{\omega})$.\\

\paragraph{Simultaneous approximation}
Consider the case of simultaneous approximation. Fix the three parameters $1 \geq \hat{\lambda}\geq 1/n$, $\rho=G(n,\hat{\lambda})$ and $c \geq 1$. Consider the Roy-system $\boldsymbol{P}$ on the interval $[1,c\rho]$ depending on these parameters whose combined graph is given below by Figure \ref{PGN2}, where
\[ P_{n+1}(1) = \cfrac{\hat{\lambda}}{1+\hat{\lambda}}, \; \; P_{k}(1) = \rho^{n-k}P_1(1) \textrm{ for } 1\leq k \leq n \textrm{ and } P_{k}(c\rho)=c\rho P_{k}(1) \textrm{ for } 2\leq k \leq n+1.\]

\begin{figure}[!h] 
 \begin{center}
 \begin{tikzpicture}[scale=0.6]

\draw (16.5,3.5) node {$\cdots$};
\draw (19.5,3.5) node {$\cdots$};

\draw (16.5,8.5) node {$\cdots$};
\draw (19.5,8.5) node {$\cdots$};

\draw (16.5,12.5) node {$\cdots$};
\draw (19.5,12.5) node {$\cdots$};

\draw[black, semithick] (9,9)--(14,14) node [above,black] {};
\draw[black, semithick] (14,10)--(15,11) node [left,black] {};
\draw[black, semithick] (17, 7)--(18,8) node [left,black] {};
\draw[black, semithick] (18, 5)--(19,6) node [left,black] {};
\draw[black, semithick] (20, 2)--(21,3) node [left,black] {};
\draw[black, semithick] (21, 1)--(34,14) node [left,black] {};

\draw[black, semithick] (16,1)--(9,1) node [left,black] {$\rho^{-(n-1)}\beta$};
\draw[black, semithick] (17,1)--(19,1) node [left,black] {};
\draw[black, semithick] (20,1)--(21,1) node [left,black] {};

\draw[black, semithick] (16,2)--(9,2) node [left,black] {$\rho^{-(n-2)}\beta$};
\draw[black, semithick] (17,2)--(19,2) node [left,black] {};

\draw[black, semithick] (16,5)--(9,5) node [left,black] {$\rho^{-k+1}\beta$};
\draw[black, semithick] (17,5)--(18,5) node [left,black] {};

\draw[black, semithick] (16,7)--(9,7) node [left,black] {$\rho^{-k}\beta$};

\draw[black, semithick] (14,14)--(16,14) node [left,black] {};
\draw[black, semithick] (17,14)--(19,14) node [left,black] {};
\draw[black, semithick] (20,14)--(34,14) node [right,black] {$c\rho \beta$};

\draw[black, semithick] (16,11)--(15,11) node [left,black] {};
\draw[black, semithick] (17,11)--(19,11) node [left,black] {};
\draw[black, semithick] (20,11)--(34,11) node [right,black] {$c \beta $};

\draw[black, semithick] (20,8)--(34,8) node [right,black] {$c\rho^{-k+1}\beta$};
\draw[black, semithick] (19,8)--(18,8) node [left,black] {};

\draw[black, semithick] (20,6)--(34,6) node [right,black] {$c\rho^{-k+2} \beta$};

\draw[black, semithick] (21,3)--(34,3) node [right,black] {$c\rho^{-(n-2)}\beta$};

\draw[black, semithick] (14,10)--(9,10) node [left,black] {$\beta$};

\draw[dashed, black] (10,10)--(10,0) node [below] { $1$};
\draw[dashed, black] (34,14)--(34,0) node [below] { $c\rho$};
\draw[dashed, black] (21,14)--(21,0) node [below] { $q_1$};
\draw[dashed, black] (14,14)--(14,0) node [below] { $q_0$};

 \end{tikzpicture}
 \end{center}
 \caption{Pattern of the combined graph of $\boldsymbol{P}$ on the fundamental interval $[1,c\rho]$, where $\beta=\frac{\hat{\lambda}}{1+\hat{\lambda}}$. }\label{PGN2}
 \end{figure}
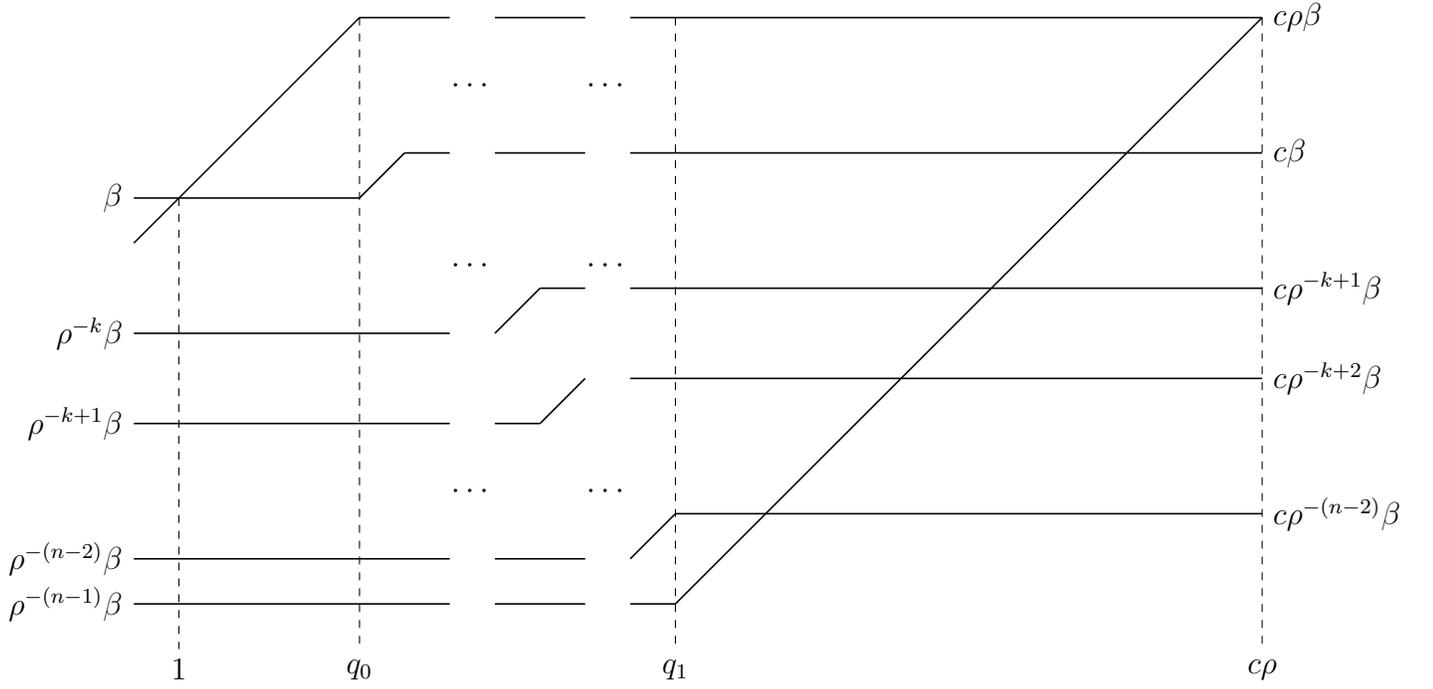

The fact that all coordinates sum up to $1$ for $q=1$ follows from $\rho$ being the root of the polynomial $R_{n,\hat{\lambda}}$ defined in \eqref{poly}. Up to change of origin and rescaling, this is the same pattern as shown by Figure \ref{PGN1}. We extend $\boldsymbol{P}$ to the interval $[1,\infty)$ by self-similarity. This means, $\boldsymbol{P}(q)=(c\rho)^{m}\boldsymbol{P}((c\rho)^{-m}q)$ for all integers $m$. In view of the value of $\boldsymbol{P}$ and its derivative at $1$ and $c\rho$, one sees that the extension provides a Roy-system on $[1,\infty)$.\\

 For $c=1$, the parameters $q_0$ and $q_1$ coincide and we constructed a \emph{regular system}.\\

Roy's Theorem \cite{Roy} provides the existence of a point $\boldsymbol{\theta}$ in $\mathbb{R}^{n}$ such that 
\begin{equation*}
\cfrac{\hat{\lambda}(\boldsymbol{\theta})}{1+\hat{\lambda}(\boldsymbol{\theta})} =  \liminf_{q\to+\infty} \cfrac{P_{n+1}(q)}{q} \;\; \textrm{ and } \;\;
 \cfrac{\lambda(\boldsymbol{\theta})}{1+\lambda(\boldsymbol{\theta})} = \limsup_{q\to+\infty} \cfrac{P_{n+1}(q)}{q}.
\end{equation*}

Again, self-similarity ensures that the $\limsup$ (resp. $\liminf$) is in fact the maximum (resp. the minimum) on the interval $[1,c\rho[$. Thus,
\begin{align*}
\cfrac{\hat{\lambda}(\boldsymbol{\theta})}{1+\hat{\lambda}(\boldsymbol{\theta})} &=  \min_{[1,c\rho[} \cfrac{P_{n+1}(q)}{q} = \cfrac{P_{n+1}(1)}{1} = \cfrac{\hat{\lambda}}{1+ \hat{\lambda}},\\
 \cfrac{\lambda(\boldsymbol{\theta})}{1+\lambda(\boldsymbol{\theta})} &= \max_{[1,c\rho[} \cfrac{P_{n+1}(q)}{q} = \cfrac{P_{n+1}(q_0)}{q_0} = \cfrac{c\rho\hat{\lambda}}{1+ c\rho\hat{\lambda}}
\end{align*}
where \[q_0 = \cfrac{\hat{\lambda}(c\rho + (1 + \rho^{-1} + \cdots + \rho^{-(n-1)}))}{1+ \hat{\lambda}} = \cfrac{1+c \rho\hat{\lambda}}{1+\hat{\lambda}}.\] 
The simplification comes from $\rho$ being the root of the polynomial $R_{n,\hat{\lambda}}$ defined in \eqref{poly}. Hence, $ \hat{\lambda}(\btheta) = \hat{\lambda}$ and $\lambda(\btheta) = c\rho \hat{\lambda}$, and we constructed the required points since $c\geq1$ and $\rho=G(n,\hat{\lambda})$.\\

For both simultaneous approximation and approximation by a linear form, the constructed 3-parameters families of self-similar Roy-systems provide infinitely many distinct points $\btheta\in \bbR^n$ via Roy's theorem with $\mathbb{Q}$-linearly independent coordinates with $1$, as explained in \cite[end of \S3]{Mar}. The $\bbQ$-linear independence comes from $P_1(q)\to \infty$ when $q \to \infty$. The construction of infinitely many points follows from a change of origin with the same pattern and self-similarity. The degenerate cases when some of the exponents are infinite is managed by (non self-similar) Roy-systems consisting in patterns described by Figure \ref{PGN1} or \ref{PGN2}, where the parameter $c$ and/or $\hat{\lambda}$ or $\hat{\omega}$  increases to infinity at each repetition. An explicit example of this trick is to be found in \cite[end of \S3]{Mar}.\qed\\

\paragraph{Aknowledgment}
We are very grateful for the hospitality of Mathematisches Forschungsinstitut Oberwolfach. An important part of this work has been done during Research in Pairs stay 1823r.  We  are also 
very grateful for the hospitality of the 
Centro Internazionale per la Ricerca Matematica (C.I.R.M.) of Trento, as the last part of the work was done during Research in Pairs stay 
there during May 19-June 1, 2019.

\bibliographystyle{amsplain}

\end{document}